\documentclass[11pt]{amsart}
\usepackage[left=2cm,top=2cm,right=3cm,nofoot]{geometry}
\usepackage{amsmath,mathtools}%
\usepackage{amsfonts}%
\usepackage{amssymb}%
\usepackage{graphicx}
\usepackage{esint}
\usepackage{hyperref}
\usepackage{enumitem}
\usepackage{accents}
\usepackage{etoolbox}
\usepackage{color} %
\patchcmd{\subsection}{-.5em}{.5em}{}{}
\newtheorem{theorem}{Theorem}[section]
	\newtheorem*{gursky}{Theorem [Gursky and Malchiodi]}
	\newtheorem*{bonheu}{Theorem [Bonheure and Nascimento]}
\theoremstyle{plain}

\newtheorem{corollary}[theorem]{Corollary}

\newtheorem{definition}[theorem]{Definition}

\newtheorem{lemma}[theorem]{Lemma}

\newtheorem{proposition}[theorem]{Proposition}
\newtheorem{remark}[theorem]{Remark}

\newcommand{\ricc}[0]{\ensuremath{\operatorname{Ric}}}
\newcommand{\ep}[0]{\ensuremath{\varepsilon}}
\newcommand{\geph}[0]{\ensuremath{g+\varepsilon^2 h}}
\newcommand{\RR}[0]{\ensuremath{\mathbb{R}}}
\numberwithin{equation}{section}
\theoremstyle{plain}

\usepackage{etoolbox}
\AtEndEnvironment{proof}{\setcounter{claim}{0}}

\newcommand{\scal}{\operatorname{scal}}
\newcommand{\diver}{\operatorname{div}}

\newcommand{\g}{\mathfrak{g}}

\newcommand{\re}{\mathbb{R}}

\begin{document}
\title[Constant $Q$-curvature]{Multiplicity results for constant $Q$-curvature conformal metrics}

\author{Salom\'{o}n Alarc\'on}
\address{Universidad T\'ecnica Federico Santa Mar\'ia, Av. España 1680, 2340000 Valpara\'iso, Valpara\'iso, Chile}
\email{salomon.alarcon@usm.cl}

\author{Jimmy Petean}
\address{Centro de Investigaci\'{o}n en Matem\'{a}ticas, CIMAT, Calle Jalisco s/n, 36023 Guanajuato, Guanajuato, M\'{e}xico}
\email{jimmy@cimat.mx}

\author{Carolina Rey}
\address{Universidad T\'ecnica Federico Santa Mar\'ia, Av. España 1680, 2340000 Valpara\'iso, Valpara\'iso, Chile}
\email{carolina.reyr@usm.cl}

\thanks{
	\mbox{\hspace{11pt}}{\it Key words}: Paneitz-Branson type equations, Q-curvature, Lusternik-Schnirelmann
	category\\
	S. Alarcón is partially supported by FONDECYT grant 1211766\\
	J. Petean is supported by  Fondo Sectorial SEP-CONACYT grant	A1-S-45886\\
	C. Rey is supported by FONDECYT grant 3200422.}

\begin{abstract} We prove that certain subcritical Paneitz-Branson type equations on a closed Riemannian manifold
$(M,g)$ have at least $\mathrm{Cat}(M)$ positve solutions, where $\mathrm{Cat}(M)$ is the Lusternik-Schnirelmann
category of $M$. This implies that if $(X,h)$ is a closed positive Einstein manifold then for $\ep >0$ small enough
there are at least $\mathrm{Cat}(M)$ metrics of constant $Q$-curvature in the conformal class of the
Riemannian product $g+\ep h$.
\end{abstract}
\maketitle

\section{Introduction}
Consider a closed Riemannian manifold $(\mathbf{M}, \g)$ of dimension $N \geq 5$. 
The Paneitz operator $P_{\g}$ is defined using a local $\g$-orthonormal frame $\left(e_{i}\right)_{i=1}^{n}$, as
\begin{equation}\label{opPan}
P_{\g} \psi:=\Delta_{\g}^{2} \psi+\frac{4}{N-2} \diver_{\g}\left(\ricc_{\g}\left(\nabla \psi, e_{i}\right) e_{i}\right)-\frac{N^{2}-4 N+8}{2(N-1)(N-2)} \diver_{\g}\left(\scal_{\g} \nabla \psi\right)+\frac{N-4}{2} Q_{\g} \psi ,
\end{equation}
where $\Delta_{\g} u=\diver_{\g}(\nabla_{\g} u)$ is the (negative) Laplace operator on $(\mathbf{M}, \g)$, $\diver_{\g}$ is the divergence, $\scal_{\g}$ is
the scalar curvature and $\mathrm{Ric}_{\g}$ is the Ricci tensor.  $Q_{\g}$ is the $Q$-curvature of $\g$, which is given by:
\begin{equation}\label{Qcurv}
Q_{\g}:=-\frac{1}{2(N-1)} \Delta_{\g} \scal_{\g}-\frac{2}{(N-2)^{2}}\left|\ricc_{\g}\right|^{2}+\frac{N^{3}-4 N^{2}+16 N-16}{8(N-1)^{2}(N-2)^{2}} \scal_{\g}^{2}.
\end{equation}
The Paneitz operator and the related $Q$-curvature were introduced by Stephen Paneitz  \cite{Paneitz} and by Thomas Branson \cite{Branson}. 
We  denote by $\left[\g_{0}\right]$ the conformal class of a given metric $\g_{0}$ in $\mathbf{M}$. 
For practical reasons, we usually write a  conformal metric $\g \in\left[\g_{0}\right]$ as $\g=u^{\frac{4}{N-4}} \g_{0}$, for a positive function $u: \mathbf{M} \rightarrow \mathbb{R}$. 
The Paneitz operator is conformally invariant in the sense that for any $\psi: \mathbf{M} \rightarrow \mathbb{R}$,
\begin{equation*}
P_{\g} \psi=u^{-\frac{N+4}{N-4}} P_{\g_{0}}(u \psi).
\end{equation*}
It follows that we can express the $Q$-curvature of $\g$ in terms of 	$\g_{0}$ and $u$ by:
\begin{equation*}
Q_{\g}=\frac{2}{N-4} P_{\g}(1)=\frac{2}{N-4} u^{-\frac{N+4}{N-4}} P_{\g_{0}}(u).
\end{equation*}
Similarly to the Yamabe problem of finding metrics of constant scalar curvature in a conformal class, there has been great interest in studying the problem of finding metrics of constant $Q$-curvature in a conformal class.

It follows from the previous comments that finding a conformal metric $\g=u^{\frac{4}{n-4}} \g_{0}$ with constant $Q$-curvature amounts to finding a positive function $u$ which solves the fourth order equation
\begin{equation}\label{ConstantQcurv}
P_{\g_{0}} u=\lambda u^{\frac{N+4}{N-4}}, \quad \lambda \in \mathbb{R} .
\end{equation}
We   call this equation the Paneitz-Branson Equation. 	

Note that the $Q$-curvature is defined if $\mathbf{M}$ has dimension greater than 2. 
However, when the dimension is 3 or 4, the constant curvature equation has a different form.  
Some significant results obtained in the $4$-dimensional case are given, for instance, by Sun-Yung A. Chang and Paul C. Yang \cite{Chang}, by Simon Brendle \cite{Brendle} and by Zindine Djadli and Andrea Malchiodi \cite{Djadli}. 
In this work, we restrict our discussion to the case when the dimension is at least 5. 

Previous research has considered the problem of  the existence of metrics of constant $Q$-curvature in dimension $N\geq 5$. 
For example, Jie Qing and David Raske in \cite{Qing} proved the existence of a positive solution of the Paneitz-Branson equation on a locally conformally flat manifold of positive scalar curvature and Poincar\'{e} exponent less than $\frac{N-4}{2}$. 
Meanwhile, Emmanuel Hebey and Fr\'{e}d\'{e}ric  Robert in \cite{Hebey} also proved the existence of a positive solution for locally conformally flat manifolds of positive scalar curvature, assuming that the Paneitz operator and the corresponding Green function were positive. 
Matthew J. Gursky and Andrea Malchiodi  in \cite{Malchiodi} proved the existence of a positive solution under the conditions  that $Q_{\g}$ is non-negative and positive somewhere in $M$ and $\scal_{\g} \geq 0$.

\medskip

Solutions of the Paneitz-Branson equation are the critical points of the functional $\mathcal{P}_{\g}: 
H^2 (\mathbf{M})\setminus \{ 0 \} \rightarrow \re$, given by

$$
\mathcal{P}_{\g} (u) := 
\frac{2}{N-4} \frac{\int_{\mathbf{M}} u\,  P_{\g} (u) \ dv_g }{\left( \int_{\mathbf{M}} |u|^{\frac{2N}{N-4}} dv_g \right)^{\frac{N-4}{N}}  }.
$$

Now consider the constant
$$
Y_4 (\g )  = \inf_{u\in C^\infty(M) , u>0}   \mathcal{P}_{\g} (u) .
$$
Fengbo Hang and Paul C. Yang proved in \cite{Yang} that if $Y_4 (\g)$ is positive, and also $Q_\g >0$ and $\scal_\g >0$ (under some weaker hypotheses, it is only needed that the Yamabe invariant of the conformal class of  $\g$ is positive and that $Q_{\g} \geq 0$ and positive somewhere) then
$Y_4 ( \g) \leq Y_4 (\g^N_0 )$, where $\g_0^N$ is the round metric on the $N$-sphere, and the infimum is always realized by a positive function, which is a solution of the Paneitz-Branson equation.
Other related existence results can be found in \cite{Hebey, Djadli, Esposito, Gursky, Hang, Robert}.  

In the case of the round sphere $(\mathbb{S}^N, \g_0 $), there is a non-compact family of solutions to the Paneitz-Branson equation that comes from the family of conformal diffeomorphisms of the sphere. 
These are all the solutions of the Paneitz-Branson equation on the sphere by a theorem of Chang-Shou Lin \cite{Lin}. 

There are also important results about the compactness of the space of solutions, for instance see \cite{Hebey2, Li, YanYanLi, Malchiodi2}. 
Examples of non-compactness of the space of positive solutions, different from the constant curvature metric on the sphere, were constructed in high dimensions  by Juncheng  Wei and Chunyi  Zhao \cite{Wei}.
Naturally, if $\g$ is Einstein it has constant $Q$-curvature. 
Moreover, J\'{e}r\^{o}me V\'{e}tois \cite{Vetois} proved that if $\g$ is Einstein and it is not isometric to the constant curvature metric on the sphere, then it is, up to multiplication by a constant, the only metric of constant $Q$-curvature in its conformal class. 
This result is an analogue of the classical due to Morio  Obata \cite{Obata} about constant scalar curvature.

There are a few other cases where one can show the non-uniqueness of conformal metrics of constant $Q$-curvature.
One case, which also appears for the Yamabe problem, is that of Riemannian products.
If $(M,g)$ and $(X,h)$ are closed Einstein manifolds with positive Einstein constant, of dimension at least $3$ then
for any $\delta >0$, the Riemannian product $(M\times X , g+\delta h)$ has constant positive $Q$-curvature and constant positive scalar curvature. 
It is easy to compute that as $\delta \rightarrow 0$ or $\delta \rightarrow \infty$ we have that $\mathcal{P}_{g+\delta h} (1) \rightarrow \infty$. 
We are in the conditions of the result mentioned above \cite{Yang}, and therefore in the conformal class of $g+\delta h$, there must be at least one other metric of constant $Q$-curvature, which minimizes $\mathcal{P}_{g+\delta h}$.
Note that a similar argument works for the Yamabe problem, but in that case, one only needs that the scalar curvatures  are constant and positive. 

Renato Bettiol, Paolo Piccione and Yannick Sire \cite{Bettiol} considered appropriate Riemannian submersions and  reduce the equation to basic functions (functions which are constant along the fibers of the submersion). 
Using bifurcation theory for this reduced equation on the canonical variation obtained by varying the size of the fibers (under certain conditions the total spaces of these variations have constant $Q$-curvature), they prove the existence of bifurcation instants, so proving the existence of metrics in the family for  which there are other conformal metrics of constant $Q$-curvature.

\medskip

This paper aims to provide a lower bound for the number of metrics of constant $Q$-curvature for a  Riemannian product $(M\times X, g+\delta h)$ within the family of conformal metrics to $g+\delta h$. 
Using the Lusternik-Schnirelmann category, we   prove the following multiplicity result.
\begin{theorem}\label{mainThm} 
Let $(M ,g)$ be a closed $n$-dimensional Riemannian manifold and $(X,h)$ a closed $m$-dimensional Einstein manifold with Einstein constant $\Lambda_0>0$.
Consider the Riemannian product $(M\times X, \g_{\ep} := g+\ep^2 h )$. 
Assume that $m\geq 3$ or, if $m=2$, that $n\geq 7$. 
Then there exists $\ep_0 >0$ such that for any $\ep \in (0,\ep_0 )$ the Paneitz-Branson equation on $(M\times X, \g_{\ep} )$ has at least $\mathrm{Cat}(M)$ positive solutions.
\end{theorem}
Here $\mathrm{Cat}(M)$ denotes the Lusternick-Schnirelmann category of $M$, a topological invariant defined as the minimum number of contractible open sets needed to cover it. 
If $M$ is a manifold, $\mathrm{Cat}(M)$ coincides with the minimal number of critical points among all smooth scalar maps $M\rightarrow \re$ (see \cite{Takens}).
The solutions given in the theorem are functions on $M$.  
We will see in Section \ref{section2} that a function $u: M \rightarrow (0,\infty)$ solves the
Paneitz-Branson equation on $(M\times X, \g_{\ep} )$ if and only if 
\begin{equation}\label{mainEq1}
 \Delta^2_g u - \frac{N^{2}-4 N+8}{2(N-1)(N-2)}(m\Lambda_0 \varepsilon^{-2} ) \Delta_{g} u  +\frac{N-4}{2}  \ep^{-4} A_{N,m} u  +\varphi (u) +  \frac{N-4}{2} ( f_0 + \ep^{-2} f_2 ) u =  u^{\frac{N+4}{N-4}}
\end{equation}
where $N=n+m$ is the dimension of the product manifold $M\times X$, $\varphi (u)$ is a second order operator on $u$ which verifies 
$$\int_M \varphi (u) \,u \,dv_g \leq C \int_M | \nabla u |^2 dv_g ,$$
for some constant $C$, $A_{N,m}$ is a positive constant that depends on $N,m$, and  $f_0$ and $f_2$ are 
certain functions on $M$.

\medskip

In general, for  a closed $n$-dimensional Riemannian manifold $({M},g)$ we can consider an equation of the form
\begin{equation*}
\Delta_g^2 u +f(x) \Delta_g u + \phi (u) +g(x)u  = h(x) u^q  ,
\end{equation*}
\noindent
for functions $f, g , h$ on M, $q\in (1,\infty )$, and a second order operator $\phi$.
 This equation is called critical if $q=\frac{n+4}{n-4}$, subcritical if  $q < \frac{n+4}{n-4}$ and supercritical if  $q > \frac{n+4}{n-4}$ (if $n\leq 4$ it is considered subcritical for any $q>1$).  
For instance,  in \cite[Section 2]{Hebey} the authors study the critical case 
when $f, g$ are constant, $\phi =0$.

Consider the equation
\begin{equation}\label{constant}
\Delta_g^2 u -b\ep^{-2} \Delta_g u +  a \ep^{-4}u  = u^q \quad \text{on } M
\end{equation}
Assume that $a,b >0$. 
For $\ep$ small Equation \eqref{mainEq1} is, in certain sense, close to Equation \ref{constant}. 
It will be easy to check along the proof of Theorem \ref{mainThm}  that it also gives a proof of
\begin{theorem}\label{mainThm2}
If $b^2 \geq 4 a$ and $q<\frac{n+4}{n-4}$ (if $n>4$),  then there exists $\ep_0 >0$ such that for any 
$\ep \in (0,\ep_0 )$, Equation \eqref{constant} has at least $\mathrm{Cat}(M)$ positive solutions. 
\end{theorem}
In Section \ref{section2} we will give the explicit formula for the Paneitz operator $P_{ \g_{\ep} }$. In
Section \ref{section3} we will see that the positive solutions of the Paneitz-Branson equation for 
$\g_{\ep}$ are precisely the critical points of the functional 
$$
J_{\ep,q}^{g,+} (u) =
\frac{1}{\ep^n}  \int_M\!  \Big( \frac{\ep^4 }{2}  (\Delta_g u )^2 + \ep^2 \frac{b_{N,m}}{2} | \nabla u |^2 
+\frac{a_{N,m} }{2}  u^2 +\ep^4 \varphi (u) u +\frac{\ep^2}{2} (  \ep^2 f_0(x)+f_2(x) ) u^2 - 
\frac{ (u^+ )^{q+1}}{q+1}\Big) dv_g\!
$$
restricted to  the associated Nehari manifold $\mathcal{N}^g_{\varepsilon}$. In the functional
$u^+ = \max \{ u , 0 \}$.
We will also show that $
J_{\ep,q}^{g,+}$ restricted to $\mathcal{N}^g_{\varepsilon}$ is bounded from below and satisfies the Palais-Smale
condition. By Lusternik-Schnirelmann theory it follows that for any $d$ the number of positive solutions of the Paneitz-Branson
equation is at least $\mathrm{Cat}(   \Sigma_{\ep , d}  )$, 
where $\Sigma_{\ep , d} = \{ u \in \mathcal{N}^g_{\varepsilon} : 
J^{g,+}_{\varepsilon,q} (u) \leq d \}$. Let 
$$
{\bf m}_{\ep} := \inf_{u \in \mathcal{N}^g_{\ep} } J^{g,+}_{\ep , q} (u) .
$$

The rest of the article is devoted to prove that for $\ep , \delta$ small enough we have that 
$\mathrm{Cat}(   \Sigma_{\ep ,{\bf m}_{\ep} + \delta }  )   \geq \mathrm{Cat}(M)$. 
To prove this inequality we will use a method introduced by Vieri Benci and Giovanna Cerami in \cite{BenciCerami} to study multiplicity results for solutions of nonlinear elliptic problems in Euclidean domains. 
The method was later used also to study similar problems in Riemannian manifolds, for instance in \cite{Benci}. 
It is sometimes called the photography method, see for instance \cite{N} where there is a very clear and general description of the method.

In Section \ref{section4} we study the limit equation and functionals on  $\re^n$. 
A positive, minimizing, solution of this limit equation is used later to construct, for any $x\in M$ a function in 
${\bf p} (x) \in   \Sigma_{\ep ,{\bf m}_{\ep} + \delta }$, which is concentrated around $x$.  

In Section \ref{section6} we show that  for $\ep , \delta$ small enough, functions $u\in \Sigma_{\ep ,{\bf m}_{\ep} + \delta }$ are concentrated around some point. 
This is used in Section \ref{section7} to show that there is a well defined center of mass function ${\bf b}: \Sigma_{\ep ,{\bf m}_{\ep} + \delta } \rightarrow M$. 
It is easily checked that  ${\bf b} \circ {\bf p} : M \rightarrow M$ is close, and therefore  homotopic to the identity. This implies the inequality $\mathrm{Cat}(   \Sigma_{\ep ,{\bf m}_{\ep} + \delta }  )   \geq \mathrm{Cat}(M)$.

\section{Constant $Q$-curvature on a product manifold}\label{section2}

In this section, we study the Paneitz operator and the $Q$-curvature in the case of a product manifold. 
Here we establish the necessary hypotheses to apply the techniques mentioned above.

Let $(M,g)$ be a closed $n$-dimensional Riemannian manifold  and $(X,h)$ a closed $m$-dimensional Einstein  manifold  with positive Einstein constant $\Lambda_0$, {\it i.e.}, its scalar curvature is $\scal_h =m \Lambda_0$. 
We are interested in positive solutions of the constant $Q$-curvature equation \eqref{ConstantQcurv} for the product manifold  $(M\times X, \mathfrak{g}_{\varepsilon} :=g+\ep^2 h)$, which is given by
$$
P_{\mathfrak{g}_{\varepsilon}} u=\lambda_\ep u^{\frac{n+m+4}{n+m-4}}, \quad \lambda_\ep \in \RR ,
$$
where $P_{\mathfrak{g}_{\varepsilon}}$ is the Paneitz operator defined in (\ref{opPan}).  
We  consider the case $ \lambda_\ep >0$ and so we can renormalize the previous equation as 
$$
P_{\mathfrak{g}_{\varepsilon}} u=   u^{\frac{n+m+4}{n+m-4}} .
$$
We put $N:=n+m$ and compute the $Q$-curvature for the product metric. 
\begin{lemma}\label{Product$Q$-curv}Let $(M,g)$ be a closed $n$-dimensional Riemannian manifold and let $(X,h)$ be a closed $m$-dimensional Einstein manifold with positive Einstein constant $\Lambda_0$. 
If $N=n+m$, then the  $Q$-curvature for the product manifold  $(M\times X, \mathfrak{g}_{\varepsilon})$ is
\begin{equation*}
\begin{aligned}
Q_{\mathfrak{g}_{\varepsilon}}&= -\frac{1}{2(N-1)} \Delta_{g} \scal_{g} -\frac{2}{(N-2)^{2}}  \left|\mathrm{Ric}_{g}\right|^{2} +\frac{N^{3}-4 N^{2}+16 N-16}{8(N-1)^{2}(N-2)^{2}} \scal_{g}^{2} \\
&\quad+  \frac{1}{\ep^2}  \ \frac{N^{3}-4 N^{2}+16 N-16}{8(N-1)^{2}(N-2)^{2}} 2m \scal_{g} \Lambda_0 +  \frac{1}{\ep^4} \left( -\frac{2 m\Lambda_0^2}{(N-2)^{2}}  +\frac{N^{3}-4 N^{2}+16 N-16}{8(N-1)^{2}(N-2)^{2}}  m^2 \Lambda_0^2 \right).
\end{aligned}
\end{equation*}
\end{lemma}
\begin{proof} 
By using (\ref{Qcurv}) and  the fact that
$$ 
\left|\mathrm{Ric}_{\mathfrak{g}_{\varepsilon}}\right|^{2} = \left|\mathrm{Ric}_{g}\right|^{2} + \frac{1}{\ep^4} \left|\mathrm{Ric}_{h}\right|^{2}  =  \left|\mathrm{Ric}_{g}\right|^{2} + \frac{1}{\ep^4} m\Lambda_0^2,
$$
straightforward calculations leads to the claim. 
\end{proof}

Now, calling
\begin{equation}
	\begin{aligned}
f_0 &:=  -\frac{1}{2(N-1)} \Delta_{g} \scal_{g} -\frac{2}{(N-2)^{2}}\left|\mathrm{Ric}_{g}\right|^{2} +
\frac{N^{3}-4 N^{2}+16 N-16}{8(N-1)^{2}(N-2)^{2}}  \left( \scal_{g}^{2}  \right) , \label{f0f2}\\ 
f_2 &:= \frac{N^{3}-4 N^{2}+16 N-16}{4(N-1)^{2}(N-2)^{2}} m \scal_{g} \Lambda_0 , \end{aligned}
\end{equation}
and
\begin{equation}\label{ANm}
A_{N,m} :=\frac{m\Lambda_0^2}{(N-2)^2}  \left(  \frac{N^{3}-4 N^{2}+16 N-16}{8(N-1)^{2}} m -2 \right), 
\end{equation}
we get
$$
Q_{\mathfrak{g}_{\varepsilon}} = f_0 + \ep^{-2} f_2 + \ep^{-4} A_{N,m} .
$$
We get the following simple computation:
\begin{lemma}\label{A}
If $m=2$ then $A_{N,2} <0$ for $N=5, 6, 7,8$ and $A_{N,2} >0$ for $N\geq 9$. 
If $m\geq 3$, then $A_{N,m}>0$. 
\end{lemma}
\begin{proof} 
Since we have a closed $m$-dimensional Einstein manifold with positive Einstein constant $\Lambda_0$, we must have $m\geq 2$.

Assume first that $m=2$. From (\ref{ANm}), we have that $A_{N,2} >0$ if and only if 
$$
N^3 -4N^2 +16N-16 >8(N-1)^2.
$$
This inequality clearly holds if $N\geq 9$, whereas if $N=5,6,7,8$, we get  
$A_{N,2} <0$. 

Now, let $m\geq3$. Then, $A_{N,m}>0$ if and only if 
$$
m( N^3 -4N^2 +16N-16)  >16(N-1)^2.
$$
The expression on the left is clearly increasing with $m$ and the lemma follows from straightforward calculations.
\end{proof}

By restricting our study to functions that depend only on the first factor, $u:M\rightarrow \RR$, we obtain the following expression for the restricted Paneitz operator: 
\begin{equation*}
P_{\mathfrak{g}_{\varepsilon}} u=\Delta_{g}^{2} u+\frac{4}{N-2} \diver_{g}\left(\ricc_{g}\left(\nabla u, e_{i}\right) e_{i}\right)
-\frac{N^{2}-4 N+8}{2(N-1)(N-2)} \diver_{\mathfrak{g}_{\varepsilon}}\left(\scal_{\mathfrak{g}_{\varepsilon}} \nabla u \right)+\frac{N-4}{2} Q_{\mathfrak{g}_{\varepsilon}} u 
\end{equation*}
for a $g$-orthonormal frame $\left\{e_{i}\right\}_{i=1}^{n}$. Then, since $\mathfrak{g}_{\varepsilon}=g+\varepsilon^2 h$, we get
\begin{equation}\label{Pgep}
 \begin{array}[c]{ll} P_{\geph} u
 \displaystyle=\Delta_{g}^{2} u+\dfrac{4}{N-2} \diver_{g}\left(\ricc_{g}\left(\nabla u, e_{i}\right) e_{i}\right) -\dfrac{N^{2}-4 N+8}{2(N-1)(N-2)}  \Big(   \diver_{g}\left(\scal_{g} \nabla u \right) \medskip \\
 \displaystyle\qquad\qquad\quad + \varepsilon^{-2}m\Lambda_0  \Delta_g u \Big) +\dfrac{N-4}{2} ( f_0 + \ep^{-2} f_2 + \ep^{-4} A_{N,m}  ) u.\end{array}
\end{equation}
Finally, the constant $Q$-curvature equation \eqref{ConstantQcurv} for the
product manifold  $(M\times X, g+\ep^2 h)$ for a function $u:M\rightarrow \RR$ is
\begin{equation} \label{equCQC}
\Delta^2_g u - \frac{N^{2}-4 N+8}{2(N-1)(N-2)} \frac{m\Lambda_0}{ \varepsilon^2} \Delta_{g} u  +\frac{N-4}{2}  \frac{ A_{N,m}}{\ep^4} u  +\varphi (u) +  \frac{N-4}{2} ( f_0 + \ep^{-2} f_2 ) u =  u^{\frac{N+4}{N-4}},
\end{equation}
where $\varphi$ is a linear second order operator given by
\begin{equation}\label{PhiDefinition}
	\varphi (u):=\frac{4}{N-2} \diver_{g}\left(\ricc_{g}\left(\nabla u, e_{i}\right) e_{i} \right)  -\frac{N^{2}-4 N+8}{2(N-1)(N-2)}    \diver_{g}\left(\scal_{g} \nabla u \right).
\end{equation}
Next we prove an elementary property of this operator.
\begin{lemma}
Let $\varphi$ be the operator defined in \eqref{PhiDefinition}. There exists a positive constant $C$
such that
\begin{equation*}\label{PHIineq}
	\int_M \varphi (u) \,u \,dv_g \leq C\int_M | \nabla u |^2 dv_g .
\end{equation*}
\end{lemma}
\begin{proof}
Using the divergence theorem, we have
\begin{equation*}
\int_M \varphi (u)\, u\, dv_g
\begin{array}[t]{l}
\displaystyle  =\int_M \diver_{g}\left(\frac{4}{N-2}\left(\ricc_{g}\left(\nabla u, e_{i}\right) e_{i} \right)   -\frac{N^{2}-4 N+8}{2(N-1)(N-2)}  \scal_{g} \nabla u \right)u \ dv_g\medskip\\
\displaystyle =\int_M \left(\frac{4}{N-2}g_{ij}\ricc_{g}\left(\nabla u, e_i \right)(\nabla u)^i   -\frac{N^{2}-4 N+8}{2(N-1)(N-2)}  \scal_{g} |\nabla u|^2 \right) dv_g \medskip\\
\displaystyle \leq C \int_M |\nabla u|^2 dv_g,
\end{array}
\end{equation*}
where $C=C(m, n, g)$.
\end{proof}

Let 
\begin{equation}\label{Def_b_a}
b_{N,m}:=\frac{N^{2}-4 N+8}{2(N-1)(N-2)}m\Lambda_0,\quad \text{and }\quad a_{N,m}:= \frac{N-4}{2} A_{N,m}, 
\end{equation}
with $ A_{N,m}  $ as in (\ref{ANm}).
Then, from (\ref{Pgep}), we get
\begin{equation*}
	P_{\geph} u=\Delta_{g}^{2} u -b_{N,m} \ep^{-2} \Delta_g u +a_{N,m} \ep^{-4} u +
	\varphi (u) +  \frac{N-4}{2} ( f_0 + \ep^{-2} f_2 ) u,
\end{equation*}
where $f_0$ and $f_2$ are given by (\ref{f0f2}). 
Consequently, after renormalization, we can rewrite the equation (\ref{equCQC}) as
\begin{equation} \label{mainEq}
\ep^4\Delta^2_g u - \ep^2 b_{N,m} \Delta_{g} u  +a_{N,m} u  +\ep^4\varphi (u) +  \frac{N-4}{2} ( \ep^4 f_0 + \ep^{2} f_2 ) u =  u^{\frac{N+4}{N-4}}.
\end{equation}
Observe that if $\Lambda_0>0$ then  $b_{N,m}>0$, and by Lemma \ref{A},  $a_{N,m} >0$ if $m\geq 3$ or $m=2$ and $N\geq 9$. 
We also get the following result.
\begin{lemma} Let $b_{N,m}$ and   $a_{N,m}$ be the constants defined in \eqref{Def_b_a}. Then $m\geq 3$, or $m=2$ and $N\geq 9$, imply that
	$b_{N,m}>2\sqrt{a_{N,m}}$.
\end{lemma}
\begin{proof} We have
$$
\begin{array}[t]{ll}
b_{N,m}^2-4{a_{N,m}}
\displaystyle \frac{16 N \Lambda_0^2m^2}{4(N-1)^2(N-2)^2}>0.
\end{array}
$$
The proof finishes from the previous calculation and bearing in mind that $b_{N,m}>0$ and  $a_{N,m}>0$. 
\end{proof}
\begin{remark}\label{remark1} It is crucial for the computations throughout the rest of this paper to highlight that
$$
\begin{array}[t]{ll}
\displaystyle \ep^4 \int_M   u \, P_{\geph} u \ dv_g \!&\!\displaystyle= \ep^4 \int_M  (\Delta_{g}  u)^2 dv_g  +  \ep^2 \int_M b_{N,m} | \nabla u|^2 dv_g +\int_M a_{N,m}  u^2  \ dv_g \medskip\\
&\qquad\qquad \displaystyle  + \ep^4 \int_M  u \, \varphi (u) \ dv_g +  \frac{N-4}{2}  \int_M  ( \ep^4  f_0 + \ep^{2} f_2 ) u^2 dv_g.\end{array} 
$$
It is easy to see then that for any $\delta >0$, we can pick $\ep_0 > 0$ such that for any $\ep \in (0, \ep_0 )$ and for any $u\in H^2 (M)$ we have that
$$
\begin{array}[t]{ll}
\displaystyle (1+\delta )^{-1}  \left(   \ep^4 \int_M  (\Delta_{g}  u)^2 dv_g  +  \ep^2 \int_M b_{N,m} | \nabla u|^2 dv_g
 +\int_M a_{N,m}  u^2  \ dv_g  \right) 
 \medskip\\ 
 \displaystyle 
\qquad  \leq  \ep^4  \int_M   u \, P_{\geph} u \, dv_g  \medskip\\ 
 \displaystyle 
\qquad \leq (1+\delta )   \left(   \ep^4 \int_M  (\Delta_{g}  u)^2 dv_g  +  \ep^2 \int_M b_{N,m} | \nabla u|^2 dv_g
 +\int_M a_{N,m}  u^2  \ dv_g  \right).
\end{array}
$$
\end{remark}
We  will use the previous remark in next sections to obtain estimates for $\ep^4 \int_M   u \, P_{\geph} u \ dv_g $ by using  estimates for the simpler expression $\ep^4 \int_M  (\Delta_{g}  u)^2 dv_g  +  \ep^2 \int_M b_{N,m} | \nabla u|^2 dv_g +\int_M a_{N,m}  u^2  \ dv_g  $.

\section{Paneitz functional and the variational setting}\label{section3}

In this section we establish some assumptions in order to study the problem of constant $Q$-curvature on a product manifold. 
Again, we consider  a closed $n$-dimensional Riemannian manifold  $(M,g)$ and a closed $m$-dimensional Einstein  manifold  $(X,h)$ with positive Einstein constant $\Lambda_0$, and we denote by  $P_\g$ to the Paneitz operator defined in (\ref{opPan}) and by $Q_{\g}$ to the $Q$-curvature  of $\g$ given by (\ref{Qcurv}). 

We denote by $\|\cdot\|_p$ the standard  $L^p(M)$-norm for $1\leq p\leq \infty$.
We consider $H^2(M)$ endowed with the usual norm
$$
\|u\|_{H^2(M)}^2=\|\Delta_g u\|_2^2+ \|\nabla u\|_{2}^2 + \|u\|_{2}^2,
$$
which is induced by the inner product
$$
\langle u, v\rangle=\int_{M}\Delta_g u \,\Delta_g v \,dv_g	+\int_{M}  \nabla u  \nabla v \, dv_g+\int_{M}  u\,  v \, dv_g.
$$
A maximum principle for the Paneitz operator and an affirmative answer to the problem 
of finding metrics  of constant $Q$-curvature in a closed Riemannian manifold was given by M. Gursky and
A. Malchiodi  in \cite[Theorems A, D]{Malchiodi}, and it reads as
\begin{gursky}
	Let $\left(M, \g\right)$ be a closed Riemannian manifold of dimension $n \geq 5$. Assume
\begin{itemize}
\item[(i)] $Q_{\g}$ is semi-positive,
\item [(ii)] $\scal_{\g} \geq 0$.
\end{itemize}
If $u \in C^{4}$ satisfies
$$
P_{\g} u \geq 0,
$$
then either $u>0$ or $u \equiv 0$ on $M$. There exists a conformal metric  $h=u^{\frac{4}{n-4}} \g$  
with constant positive $Q$-curvature and positive scalar curvature.
\end{gursky}

Consider the normalized total $Q$-curvature functional defined on the space of Riemannian metrics $g$ on the $n$-dimensional manifold $M$, which is given by
$$
\mathcal{Q} (g) := \frac{\int_M Q_g \,dv_g}{\mathrm{Vol}(M,g)^{\frac{n-4}{n}}} .
$$
Consider the restriction of the functional to the conformal class $[g]$ and the Paneitz functional $ \mathcal{P}_g$ given by
$$
\mathcal{P}_g (u) := \mathcal{Q} (u^{\frac{4}{n-4}} g) = 
\frac{2}{n-4} \frac{\int_M u\,  P_g (u) \ dv_g }{\left( \int_M |u|^{\frac{2n}{n-4}} dv_g \right)^{\frac{n-4}{n}}  }.
$$
Then we define the following constants as in \cite{Gursky}:
\begin{equation*}
Y_4 (M,g)  = \inf_{u\in C^\infty(M)}   \mathcal{P}_g (u) 
\quad 
\text{ and }
\quad 
Y_4^+  (M,g) = \inf_{h\in [g]} \mathcal{Q} (h) = \inf_{u>0}   \mathcal{P}_g (u).
\end{equation*}
We consider $m\geq 3$ and $n\geq 2$, or $m=2$ and $n\geq 7$, so the Riemannian product $(X\times M, \mathfrak{g}_{\ep}:=g+ \ep^2 h)$ satisfies the conditions of Lemma \ref{A}. It follows from Lemma
\ref{Product$Q$-curv} that:
\begin{lemma}
There exists $\ep_0 >0$ such that if $\ep \in (0,\ep_0 )$ then $\scal_{\mathfrak{g}_{\ep}}>0$ and ${Q}_{\mathfrak{g}_{\ep}}>0$. 
\end{lemma}

In particular, for $\ep \in (0,\ep_0 )$, we are under the conditions of the Gursky-Malchiodi Theorem, and then the equality 
$Y_4 (M,\mathfrak{g}_{\ep} )=Y_4^+  (M,\mathfrak{g}_{\ep} )$ is achieved by a metric of constant $Q$-curvature.

\medskip

For $q>1$ consider the functional $J_{\ep,q}^{g,+} : H^{2}(M) \rightarrow \mathbb{R}$ defined by
\begin{equation}\label{J_ep}\!
J_{\ep,q}^{g,+} (u)\!=
\! \frac{1}{\ep^n} \! \int_M\!  \Big( \ep^4  \frac{1}{2}  (\Delta_g u )^2 + \ep^2 \frac{b_{N,m}}{2} | \nabla u |^2 
+\frac{a_{N,m} }{2}  u^2 +\ep^4 \Phi (u) +\frac{\ep^2}{2} F(x,\ep ) u^2 - \frac{1}{q+1} (u^+ )^{q+1}\Big) dv_g\!
\end{equation}
where
$$
\Phi (u)=\varphi (u)u   \quad\text{and}\quad F(x,\ep ) =  \ep^2 f_0(x)+f_2(x), 
$$
and $f_0$ and $f_2$ are defined in (\ref{f0f2}). 

\noindent
Also consider the associated Nehari manifold $\mathcal{N}^g_{\varepsilon}$, given by
\begin{equation}\label{N_ep}
	\begin{array}[t]{lll}
		\!\!\mathcal{N}^g_{\ep} :=\!\bigg \{ u\!\in\! H^2 (M) \setminus\{ 0 \}\! : \!
		\medskip \\
		\displaystyle \quad\quad\,\,  \int_M  \!\! \big(\ep^4    (\Delta u )^2 + \ep^2  {b_{N,m}} | \nabla u |^2 +a_{N,m} u^2  +\ep^4 \Phi (u) +\ep^2 F(x,\ep ) u^2\big)\,  dv_g = \int_M   (u^+ )^{q+1} dv_g  \bigg\} \!\!
	\end{array}
\end{equation}

\begin{lemma} Assume that $q>1$ and if $n>4$ that $q<\frac{n+4}{n-4}$.
The functional $J_{\ep,q}^{g,+} \in C^{1}\left(H^{2}(M), \mathbb{R}\right)$ and
\begin{equation}\label{derivative}
\begin{aligned}
			\left\langle (J_{\ep,q}^{g,+})^\prime(u), v\right\rangle&= \!
\frac{1}{\ep^n}		\left(\!	\ep^4\!\int_{M}\Delta_g u \Delta_g v\, dv_g+b_{N,m}  \ep^2 \!\int_{M}\nabla u \nabla v\, dv_g+a_{N,m} \!\int_{M}\!u\, v \,dv_g \right)\\
			&\quad\,+\frac{1}{\ep^n}	\!	\left(	\!\ep^4 \int_M \varphi (u)v \, dv_g +\ep^2\! \int_M \left(\ep^2 f_0+f_2 \right) u\,v \, dv_g  -\int_M(u^{+})^{q-1}u\,v \,d v_g\!\right),
\end{aligned}
\end{equation}
for all $u, v \in H^2(M)$. 
Moreover, if $u \in H^2(M)$, then $(J^{g,+}_{\ep,q})^{\prime}(u) \in (H^2(M))^{*}$, where $(H^{2}(M))^*$ denotes the dual space of $H^{2}(M)$.
\end{lemma}
\begin{proof}
First, we write $J^{g,+}_{\ep,q}=\dfrac{1}{\ep^n} \left(I^1_\ep-I^2_\ep\right)$, where
$$
I_\ep^2(u)=	\frac{1}{q+1} \int_M (u^+)^{q+1}\,dv_g,
$$
and $I^1_\ep=J^{g,+}_{\ep,q}+I^2_\ep$. 
It is easy to see that for all $u, v \in H^{2}(M)$, \eqref{derivative} holds and  that $I^1_\ep \in C^{1}\left(H^{2}(M), \mathbb{R}\right)$.
	
Next, we   show $I^2_\ep \in C^{1}(H^{2}(M),\re)$.
Let $\left\{u_{j}\right\} \subset H^{2}(M)$ a sequence such that $u_{j} \rightarrow u$ for some $u \in H^{2}(M)$, strongly in $H^{2}(M)$ as $j \rightarrow \infty$.
Since $1<q + 1<2_n^{\sharp}$, there exists a subsequence of $\left\{u_{j}\right\}$ still denoted by $\left\{u_{j}\right\}$ such that $u_{j} \to u$ a.e. in $M$. 
By H\"older's inequality, we get 
\begin{eqnarray*}
\left|\left\langle  {I^2_\ep}^\prime (u_j),v\right\rangle \right|&\leq&
		\displaystyle\int_{M}\left|(u_{j}^+)^{q-1}u_jv\right| \ dv_g\\
		&\leq& \displaystyle\int_{M}|u_{j}|^{q}|v| \ dv_g\\
		&\leq&  \|u_j\|_{L^q(M)}^{1-1/q} \|v\|_{L^q(M)}^q\\
		&\leq& C \left\|u_j\right\|_{H^2(M)}\left\|v\right\|_{H^2(M)}
\end{eqnarray*}
since $H^2(M)\hookrightarrow L^q(M)$.
Additionally, the fact that $u_{j} \to u$ strongly in $H^{2}(M)$ implies
\begin{equation*}\label{cont2}
\lim _{n \rightarrow \infty} \|u_j-u\|_{H^2(M)} =0.
\end{equation*}
Then we have
$$
\left\| {I^2_\ep}^\prime(u_j) - {I^2_\ep}^\prime (u)\right\|=\sup _{v \in H^{2}(M),\|v\|_{H^2(M)} \leq 1}\left|\left\langle {I_\ep^2}^{\prime}\left(u_{j}\right)-{I_\ep^2}^{\prime}(u), v\right\rangle\right| \rightarrow 0
\qquad
\text{ as } j \rightarrow \infty.
$$
\end{proof}
From the previous lemma, it follows that positive solutions of \eqref{mainEq} belong to $\mathcal{N}^g_{\varepsilon}$ and are critical points for $J^{g,+}_{\varepsilon,q}$. 
Moreover, we get the following result.
\begin{lemma}\label{CriticalPoints} For $\ep$ small enough, the critical points of $J^{g,+}_{\varepsilon,q}$ restricted to $\mathcal{N}^g_{\ep}$ are the positive solutions of
Equation \eqref{mainEq}.
\end{lemma}
\begin{proof} 
Let $u \in \mathcal{N}^g_{\varepsilon}$ be a critical point of $J^{g,+}_{\varepsilon,q}$ restricted to $\mathcal{N}^g_{\ep}$. 
Then $u$ must be nonnegative and a weak solution of \eqref{mainEq}.
Since the embedding of $H^2(M)$ in $L^q(M)$ is compact, we know from the classical bootstrap argument that $u\in L^s(M)$ for all $s\geq 1$.  
By regularity theory (see \cite{Robert}), we get $u \in C^4(M)$.
	
Finally, by the theorem of Gursky and Malchiodi \cite[Theorem A]{Malchiodi}  that we mentioned above we know that for $\ep$ small enough, if $u \in C^{4}(M\times X)$ satisfies
$$
P_{\mathfrak{g}_\ep} u \geq 0,
$$
then either $u>0$ or $u \equiv 0$ on $M\times X$. 
In consequence, $u$ is positive and smooth.
\end{proof}
By definition, a sequence $\left(u_{i}\right)$ of functions in $\mathcal{N}^g_{\varepsilon}$ is said to be a Palais-Smale sequence for the functional $J^{g,+}_{\varepsilon,q}$ in (\ref{J_ep}) if:
\begin{itemize}
	\item[(PS1)] $J^{g,+}_{\varepsilon,q}\left(u_{i}\right)=O(1)$, when $i\to \infty$, and
	\item[(PS2)] $(J^{g,+}_{\varepsilon,q})^{\prime}\left(u_{i}\right) \rightarrow 0$ in $(H^{2}(M))^*$ as $i \to +\infty$, where $(H^{2}(M))^*$ denotes the dual space of $H^{2}(M)$.
\end{itemize}
We say that $J^{g,+}_{\varepsilon,q}$ satisfies Palais-Smale condition in $\mathcal{N}^g_{\varepsilon}$, if for any Palais-Smale sequence $\left\{u_{i}\right\} \subset \mathcal{N}^g_{\varepsilon}$, there exists a convergent subsequence of $\left\{u_{i}\right\}$.
\begin{theorem}\label{PS} Assume that $q>1$ and if $n>4$ that $q<\frac{n+4}{n-4}$. Then, $J^{g,+}_{\varepsilon,q}$ restricted to $\mathcal{N}^g_{\varepsilon}$ is bounded below and it satisfies the Palais-Smale condition for $\ep$ small enough.
\end{theorem}
\begin{proof}  
Note that if $u \in \mathcal{N}^g_{\varepsilon}$ then $J^{g,+}_{\varepsilon,q} (u) \geq 0$.
Let $\left\{u_{i}\right\}$  be a Palais-Smale sequence  in $\mathcal{N}^g_{\varepsilon}$ for $J^{g,+}_{\varepsilon,q}$. 
Then, for $\ep$ small enough, we have
$$
O(1) + o(1)\|u_i\|_{H^2(M)}  \begin{array}[t]{ll}=(q+1)J^{g,+}_{\varepsilon,q} (u_i) - \langle (J^{g,+}_{\varepsilon,q})^{\prime}(u_i), u_i\rangle 
\medskip\\
\displaystyle 
	=\frac{1}{\ep^{n}}\bigg( \frac{q-1}{2}  \int_M   \big(\ep^4    (\Delta u )^2 + \ep^2 b_{N,m} | \nabla u |^2 +a_{N,m}   u^2 \big)\,dv_g \medskip\\
	\displaystyle \qquad\qquad + \ep^4q\int_M \varphi(u)\,u\, dv_g + \frac{q-1}{2}\int_M \ep^2F(\ep, x) u^2 dv_g \bigg)
\medskip\\
\displaystyle 
	\geq C \|u_i\|^2_{H^2(M)},
\end{array}
$$
where $C=C(N,m,q,\ep)$, due to Remark \ref{remark1}.
It follows that $\|u_i\|_{H^2(M)}$ is bounded in $H^2(M)$ and in consequence, there is a subsequence, still denoted $(u_i)$ which converges weakly to a function $u\in H^2(M)$ and strongly in $H^1(M)$. 
Set $v_i=u-u_i\in H^2(M)$. 
We have
\begin{equation}\label{diference}
\begin{aligned} 
			&\ep^4\int_{M}|\Delta_g u_i|^2 \,dv_g+b_{N,m} \  \ep^2 \int_{M}|\nabla u_i|^2 \,dv_g+a_{N,m} \int_{M}u_i^2 \,dv_g\medskip  \\
			&\quad=\ep^4\int_{M} (\Delta_g u )^2 dv_g+b_{N,m} \  \ep^2 \int_{M}|\nabla u|^2 dv_g+a_{N,m} \int_{M}u^2 \,dv_g  \medskip \\
			&\quad\qquad-2\left(\varepsilon^4\int_M \Delta_g u\Delta_g v_i \,dv_g+b_{N,m} \  \ep^2 \int_{M}\nabla u\nabla v_i \,dv_g+a_{N,m} \int_{M}u v_i \,d v_g\right)\medskip\\ 
			&\quad\qquad+\ep^4\int_{M}(\Delta_g v_i)^2\,dv_g+b_{N,m} \  \ep^2 \int_{M}|\nabla v_i|^2 \,dv_g+a_{N,m} \int_{M}v_i^2\,dv_g  \\
\end{aligned}
\end{equation}
Because of the weakly convergence in $H^2(M)$, we know that
$$
o(1)=\int_{M}\Delta_g w \Delta_g v_i \,dv_g+ \int_{M}\nabla w \nabla v_i \,dv_g+\int_{M}w v_i \,dv_g, 
$$
for all $w\in H^2(M)$, and the strong convergence in $H^1(M)$ implies
$$
 \int_M(|\nabla v_i|^2+v_i^2)\,dv_g=o(1).
$$
Then, from (\ref{diference}). we get
\begin{equation*}
	\begin{aligned} 
	&\ep^4\int_{M}(\Delta_g u_i)^2  \,dv_g+b_{N,m} \  \ep^2 \int_{M}|\nabla u_i |^2  \,dv_g+a_{N,m} \int_{M}u_i^2 \,dv_g  \\
	&=\ep^4\int_{M}  (\Delta_g u)^2 \,dv_g+b_{N,m} \  \ep^2 \int_{M}|\nabla u|^2  \,dv_g+a_{N,m} \int_{M}u^2\,dv_g  +\ep^4\int_{M}(\Delta_g v_i)^2\,dv_g + o(1). \\
	\end{aligned}
\end{equation*}
Moreover, since $2 \leq q+1 < 2_n^{\sharp}$ and the embedding $H^2(M) \hookrightarrow L^q(M)$ is compact, we can assume that $\lim_{i\to+\infty} u_i = u$ in $L^q(M)$. 
Consequently, we get that
$$
\lim_{i \rightarrow +\infty}\int_M(u_i^{+})^{q+1}d v_g=\int_M(u^{+})^{q+1} d v_g.
$$
Plugging this equalities in $\langle (J^{g,+}_{\varepsilon,q})^{\prime} (u_i), v_i\rangle$, one gets that
\begin{equation*}
\int_{M}(\Delta_g v_i)^2dv_g=o(1)
\end{equation*}
which means, that the convergence in $H^2(M)$ is strong.
\end{proof}

We are now in the position to apply the well known result from Lusternik-Schnirelmann theory, which says that a $C^1$-functional  $J$ on a Banach manifold $N$ which is bounded form below and satisfies the Palais-Smale condition has at least $\mathrm{Cat} (\{ x \in N : J(x) \leq a \} )$ critical points, for any $a$
(see for instance \cite[Chapter 9]{AM} or \cite{BCP}). 

Let $\Sigma_{\ep , d} = \{ u \in \mathcal{N}^g_{\varepsilon} : J^{g,+}_{\varepsilon,q} (u) \leq d \}$. 
Then Lemma \ref{CriticalPoints} and Theorem \ref{PS} imply:
\begin{corollary}\label{cor36} There exists $\ep_0 >0$ such that for any $\ep \in (0,\ep_0 )$ and
any $d>0$ Equation \eqref{mainEq} has at least $\mathrm{Cat} (\Sigma_{\ep ,d} )$ positive solutions.
\end{corollary}

\section{The limit equation and functionals on $\re^n$ }\label{section4}

In this section we introduce the limit problem. 
Recall that in the equation $2<q+1<2_n^\sharp$, where $2_n^\sharp = \infty$ if $n\leq 4$ and  $2_n^\sharp=  \frac{2n}{n-4}$  if $n>4$.
For $\alpha, \beta >0$, consider $J_{\alpha , \beta } : H^2 (\re^n ) \rightarrow \re$,  given by
$$
J_{\alpha , \beta} (u) := \frac{1}{2} \int_{\re^n}  \left( (\Delta u )^2  + \beta  | \nabla  u |^2  +\alpha u^2 \right)dx.
$$
Let $E_{\alpha , \beta , q} : H^2 (\re^n ) \rightarrow \re$, defined by
$$
E_{\alpha , \beta , q} (u):= \frac{1}{2}  \int_{\re^n} \left( (\Delta u )^2  +
 \beta  | \nabla  u |^2  +\alpha u^2\right) dx -  \frac{1}{q+1} \int_{\re^n} |u|^{q+1} dx .
$$
Also let $Y_{\alpha , \beta , q} : H^2 (\re^n ) \setminus \{ 0 \} \rightarrow \re$, 
$$
Y_{\alpha , \beta , q} (u) :=   \frac{ J_{\alpha , \beta} (u) }{ \| u\|_{q+1}^2 }. 
$$
Consider 
$$
\mathcal{N}(\alpha , \beta , q):= \bigg\{ u \in H^2 (\re^n ) \setminus \{ 0 \}  :2 J_{\alpha , \beta} (u) =   \int_{\re^n}     | u |^{q+1} dx  \bigg\}.
$$
Let
$$
m_{\alpha , \beta}: = \inf_{u \in \mathcal{N}(\alpha , \beta , q)} E_{\alpha , \beta , q} (u)  =
 \inf_{u \in \mathcal{N}(\alpha , \beta , q)}  \frac{q-1}{2(q+1)} \| u \|_{q+1}^{q+1} \geq 0  .  
$$
Note that for any $u \in H^2 (\re^n ) \setminus \{ 0 \}$ there exists exactly one $\lambda
=\lambda (u) >0$ such that $\lambda (u)  u \in  \mathcal{N}(\alpha , \beta , q)$ and that for any constant $c \neq 0$, $Y_{\alpha , \beta , q} ( c u)   =  Y_{\alpha , \beta , q} (u) $. 
It follows that 
\begin{equation}\label{car}
\begin{aligned}
	m_{\alpha , \beta} &= \inf_{u \in \mathcal{N}(\alpha , \beta , q)}  \frac{q-1}{2(q+1)} \| u \|_{q+1}^{q+1}  \\
	&= \frac{q-1}{2(q+1)} \left(    \inf_{u \in \mathcal{N}(\alpha , \beta , q)}   \| u \|_{q+1}^{q-1} \right)^{\frac{q+1}{q-1}} \\
	&= \frac{q-1}{2(q+1)} \left(    \inf_{u \in \mathcal{N}(\alpha , \beta , q)}  2 Y_{\alpha , \beta , q} (u)  \right)^{\frac{q+1}{q-1}} \\
	& = \frac{q-1}{q+1}  \ 2^{\frac{2}{q-1}}
	\left( \inf_{u  \in H^2 (\re^n ) \setminus \{ 0 \} }  Y_{\alpha , \beta , q} (u) \right)^{\frac{q+1}{q-1}}\\
&= \frac{q-1}{q+1} \ 2^{\frac{2}{q-1}} 
\left( \inf_{u  \in H^2 (\re^n ), \| u \|_{q+1} =1 } J_{\alpha , \beta} (u)  \right)^{\frac{q+1}{q-1}}.\\
\end{aligned}
\end{equation}
Let us recall an important result from \cite[Theorem 1.1]{BN} about this infimum.
\begin{bonheu}
 If $\beta \geq 2 \sqrt{\alpha}$ then $m_{\alpha , \beta} $ is realized 
by a positive, radial (around some point), radially decreasing function.
\end{bonheu}
Assume from now on that $\beta \geq 2 \sqrt{\alpha}$ and let $U_{\alpha , \beta} \in \mathcal{N}(\alpha , \beta , q)$ be a minimizer like in the theorem.
Similarly consider
$$
E_{\alpha , \beta , q}^+ (u) =  \frac{1}{2}  \int_{\re^n} \left( (\Delta u )^2  +
 \beta  | \nabla  u |^2  +\alpha u^2 \right)dx -  \frac{1}{q+1} \int_{\re^n} (u^+ )^{q+1} dx
$$
and $Y_{\alpha , \beta , q}^+ : \{ u \in H^2 (\re^n ) : u^+ \neq 0 \} \rightarrow \re $,
$$
Y_{\alpha , \beta , q}^+  (u) =   \frac{ J_{\alpha , \beta} (u) }{ \| u^+ \|_{q+1}^2 }. 
$$
Note that if $u\in H^2 (\re^n )$ and  $u^+ \neq 0$, then $Y_{\alpha , \beta , q}^+  (u) \geq Y_{\alpha , \beta , q}  (u) $.

Let 
$$
\mathcal{N}^+ (\alpha , \beta , q)= \bigg\{ u \in H^2 (\re^n ) \setminus \{ 0 \}  : 2  J_{\alpha , \beta} (u) =   \int_{\re^n}     (u^+ )^{q+1} dx \bigg \}.
$$
In the same way as before, we can see that for any $u \in H^2 (\re^n ) $ such that $u^+ \neq 0$ there exists exactly one $\lambda^+ =\lambda^+  (u) >0$ such that $\lambda^+  (u)  u \in  \mathcal{N}^+ (\alpha , \beta , q)$ and that for any positive constant $c$, $Y^+_{\alpha , \beta , q} ( c u)   =  Y^+_{\alpha , \beta , q} (u) $.
If we define
$$
m^+_{\alpha , \beta} := \inf_{u \in \mathcal{N}^+ (\alpha , \beta , q)} E_{\alpha , \beta , q}^+(u)  =
 \inf_{u \in \mathcal{N}^+ (\alpha , \beta , q)}  \frac{q-1}{2(q+1)} \| u^+  \|_{q+1}^{q+1} \geq 0  ,  
$$
then
\begin{equation*}\label{sal}
\begin{aligned}
	m^+_{\alpha , \beta} &=\frac{q-1}{2(q+1)} \left(    \inf_{u \in \mathcal{N}^+ (\alpha , \beta , q)}   \| u^+  \|_{q+1}^{q-1} \right)^{\frac{q+1}{q-1}} \\
	&=\frac{q-1}{q+1} \  2^{\frac{2}{q-1}} 
	\left( \inf_{u  \in H^2 (\re^n ), \ u^+ \neq 0 }  
	Y^+_{\alpha , \beta , q} (u) \right)^{\frac{q+1}{q-1}}\\
	&= \frac{q-1}{q+1}   \  2^{\frac{2}{q-1}}  \left(  \inf_{u  \in H^2 (\re^n ), \| u^+ \|_{q+1} =1 } J_{\alpha , \beta} (u) 
	\right)^{\frac{q+1}{q-1}}.
\end{aligned}
\end{equation*}
\begin{remark}
It follows easily from the previous observations that $m_{\alpha , \beta} \leq m^+_{\alpha , \beta} $. 
But since we have a minimizer for $m_{\alpha , \beta}$ which is a positive function it follows that actually
 $m_{\alpha , \beta} = m^+_{\alpha , \beta} $.
\end{remark}

Throughout the paper, we will use the following notation: $B(0,r)$ is the ball in $\mathbb{R}^{N}$ centered at $0$ with radius $r$.

Given $\ep >0$ and $u: \re^n \rightarrow \re$, the function $u_{\ep} (x) = u (\ep^{-1} x)$ 
satisfies the following properties:
$$
\Delta u_{\ep} (x) = \ep^{-2} \Delta u (\ep^{-1} x), \quad \text{and} \quad 
| \nabla u_{\ep} (x) |^2 = \ep^{-2}  | \nabla u (\ep^{-1} x ) |^2.
$$ 
Note also that for any function $f$ we have for any $r>0$,
$$
\int_{B(0,\ep r)} f_{\ep} (x) \,dx = \ep^n \int_{B(0.r)} f(x)\,  dx,
$$
and, in consequence, 
$$
\int_{\re^n} f_{\ep} (x)\, dx = \ep^n \int_{\re^n} f(x) \, dx.
$$
Therefore 
\begin{equation*}\label{3.4}
\ep^{-n} \int_{\re^n} \left(\ep^4 (\Delta u_{\ep} )^2  + \ep^2  \beta | \nabla u_{\ep} |^2 
+ \alpha  u_{\ep} ^2 \right)dx = \int_{\re^n }\left( (\Delta u )^2 + \beta | \nabla u |^2 + \alpha u^2 \right)dx.
\end{equation*}
Now, consider the functionals
$$
J^{\ep}_{\alpha , \beta } (u) :=\frac{ \ep^{-n} }{ 2} \int_{\re^n} \left( \ep^4 (\Delta u )^2  +
 \beta \ep^2   | \nabla  u |^2  +\alpha u^2 \right)dx
$$
and
$$
E^{\ep}_{\alpha , \beta , q } (u): =\frac{ \ep^{-n} }{2} \int_{\re^n} \left(\ep^4  (\Delta u )^2  +
 \ep^2 \beta  | \nabla  u |^2  +\alpha u^2 \right)dx - \frac{\ep^{-n} }{q+1} \int_{\re^n} |u|^{q+1} dx.
$$
Consider
$$
\mathcal{N}(\alpha , \beta , q, \ep):= \bigg\{ u \in H^2 (\re^n ) \setminus \{ 0 \}  : 2J^{\ep}_{\alpha , \beta} (u) = 
\ep^{-n}  \int_{\re^n}     | u |^{q+1} dx  \bigg\}.
$$
Note that $u \in \mathcal{N}(\alpha , \beta ,q)$ if and only if $u_{\ep} \in \mathcal{N}(\alpha , \beta , q, \ep)$ and $J_{\alpha , \beta} (u) = J^{\ep}_{\alpha , \beta } (u_{\ep}),   \   E_{\alpha , \beta ,q} (u) = E^{\ep}_{\alpha , \beta , q} (u_{\ep})$. Therefore, from (\ref{car}), for any $\ep >0$, 
\begin{equation}\label{3.5}
m_{\alpha , \beta } =  \inf_{u \in \mathcal{N}(\alpha , \beta , q, \ep)} E^{\ep}_{\alpha , \beta , q} (u) .
\end{equation}
Let 
\begin{equation} \label{Ueab}
U^{\ep}_{\alpha , \beta } (x) = U_{\alpha , \beta } ({\ep}^{-1}x),
\end{equation}
where $U_{\alpha , \beta} \in \mathcal{N}(\alpha , \beta , q)$ is a minimizer like in the Theorem of Bonheure and Nascimento.
Hence, $U^{\ep}_{\alpha , \beta } \in 
 \mathcal{N}(\alpha , \beta , q, \ep)$  and realizes  $  \inf_{u \in \mathcal{N}(\alpha , \beta , q, \ep)} E^{\ep}_{\alpha , \beta , q} (u)$. 

Define now $Y_{\alpha , \beta , q}^{\ep} :  H^2 (\re^n ) \setminus \{ 0 \} \rightarrow \re$,
$$
Y_{\alpha , \beta , q}^{\ep }  (u) :=   \frac{ J^{\ep}_{\alpha , \beta} (u) }{ \ep^{\frac{-2n}{q+1}} 
 \| u \|_{q+1}^2 }.
$$
and note that from (\ref{car}) and (\ref{3.5}) we get
$$
m_{\alpha , \beta }
\begin{array}[t]{ll}
\displaystyle  =  \inf_{u \in \mathcal{N} (\alpha , \beta , q, \ep)} 
\frac{q-1}{2q+2} \ep^{-n} \|  u  \|^{q+1}_{q+1} \medskip\\
\displaystyle = \frac{q-1}{2(q+1)} \ep^{-n}  \left(    \inf_{u \in \mathcal{N}  (\alpha , \beta , q , \ep )}   \| u  \|_{q+1}^{q-1} \right)^{\frac{q+1}{q-1}} \medskip\\
\displaystyle =\frac{q-1}{q+1} 2^{\frac{2}{q-1} }  \left(    \inf_{u \in \mathcal{N} (\alpha , \beta , q , \ep )} Y^{\ep}_{\alpha , \beta , q}  (u) \right)^{\frac{q+1}{q-1}} \medskip\\
\displaystyle 
=\frac{q-1}{q+1} 2^{\frac{2}{q-1} }  \left(    \inf_{u \in H^2 (\re^n ) \setminus\{ 0 \} } 
Y^{\ep}_{\alpha , \beta , q}  (u) \right)^{\frac{q+1}{q-1}}. 
\end{array}
$$
\begin{definition}
Let $\displaystyle {\bf Y}_{\alpha , \beta , q}^{\ep }:  = 
\inf_{u \in H^2 (\re^n ) \setminus\{ 0 \} }  Y_{\alpha , \beta , q}^{\ep }  (u)$.
\end{definition}
Then we have that 
\begin{equation}\label{3.6}
{\bf Y}_{\alpha , \beta , q}^{\ep }  = 2^{\frac{-2}{q+1}}
\left( \frac{q+1}{q-1} m_{\alpha , \beta } \right)^{\frac{q-1}{q+1}},
\end{equation}
and, for any $u \in H^2 (\re^n )$,
\begin{equation}\label{3.7}
J^{\ep}_{\alpha , \beta} (u) \geq {\bf Y}_{\alpha , \beta , q}^{\ep }  \ 
 \ep^{\frac{-2n}{q+1}} 
 \| u \|_{q+1}^2 .
\end{equation}

Similarly, let
$$
E^{\ep ,+}_{\alpha , \beta , q } (u) =\frac{ \ep^{-n} }{2} \int_{\re^n} \left(\ep^4  (\Delta u )^2  +
 \ep^2 \beta  | \nabla  u |^2  +\alpha u^2\right) dx - \frac{\ep^{-n} }{q+1} \int_{\re^n}  (u^+ )^{q+1} dx.
$$

Consider 
$$
\mathcal{N}^+ (\alpha , \beta , q, \ep)= \bigg\{ u \in H^2 (\re^n ) \setminus \{ 0 \}  : 2 J_{\alpha , \beta, \ep} (u) = 
\ep^{-n}  \int_{\re^n}     ( u^+ )^{q+1} dx \bigg \}.
$$
Note that $u \in \mathcal{N}^+ (\alpha , \beta ,q)$ if and
only if $u_{\ep} \in \mathcal{N}^+ (\alpha , \beta , q, \ep)$ and $J_{\alpha , \beta} (u) =J^{\ep}_{\alpha , \beta } (u_{\ep}),  \ \ E^{+}_{\alpha , \beta ,q} (u) =E^{\ep ,+}_{\alpha , \beta , q} (u_{\ep})$, where 
$$
E^+_{\alpha , \beta,q} (u) := \frac{1}{2} \int_{\re^n}  \left( (\Delta u )^2  + \beta  | \nabla  u |^2  +\alpha u^2 \right)dx - \frac{1 }{q+1} \int_{\re^n}  (u^+ )^{q+1} dx.
 $$
Therefore, we also have that for any $\ep >0$, 
\begin{equation*}
m_{\alpha , \beta } =  \inf_{u \in \mathcal{N}^+ (\alpha , \beta , q, \ep)} E^{\ep ,+}_{\alpha , \beta , q} (u) .
\end{equation*}

Note that if  $u \in H^2 (\re^n )$, $u^+ \neq 0$,  there exists a unique $\lambda^{+, \ep} (u) >0$ such that  $\lambda^{+, \ep} (u) \, u \in \mathcal{N}^+ (\alpha , \beta ,q, \ep )$.

Let $Y_{\alpha , \beta , q}^{\ep , +} : \{ u \in H^2 (\re^n ) : u^+ \neq 0 \} \rightarrow \re $,
$$
Y_{\alpha , \beta , q}^{\ep , +}  (u): =   \frac{ J^{\ep}_{\alpha , \beta} (u) }{ \ep^{\frac{-2n}{q+1}} 
 \| u^+ \|_{q+1}^2 }.
 $$
Then we have that 
$$
m_{\alpha , \beta }
\begin{array}[t]{ll}
\displaystyle =  \inf_{u \in \mathcal{N}^+ (\alpha , \beta , q, \ep)} E^{\ep ,+}_{\alpha , \beta , q} (u)\medskip\\
\displaystyle  = \inf_{u \in \mathcal{N}^+ (\alpha , \beta , q, \ep )}  \frac{q-1}{2(q+1)}  \ep^{-n} \| u^+  \|_{q+1}^{q+1} \medskip\\
 \displaystyle =\frac{q-1}{2(q+1)} \ep^{-n}  \left(    \inf_{u \in \mathcal{N}^+ (\alpha , \beta , q , \ep )}   \| u^+  \|_{q+1}^{q-1} \right)^{\frac{q+1}{q-1}} \medskip\\
 \displaystyle =\frac{q-1}{q+1} \  2^{\frac{2}{q-1}} 
 \left( \inf_{u  \in H^2 (\re^n ), \ u^+ \neq 0 }  
Y^{\ep ,+}_{\alpha , \beta , q} (u) \right)^{\frac{q+1}{q-1}}\medskip\\
 \displaystyle = \frac{q-1}{q+1}   \  2^{\frac{2}{q-1}}  \left(  \inf_{u  \in H^2 (\re^n ), \| u^+ \|_{q+1} =1 } 
J^{\ep}_{\alpha , \beta} (u) 
\right)^{\frac{q+1}{q-1}} .
\end{array}
$$

On $\re^n$ consider a smooth radial cut-off function $\varphi_{s} : \re^n \rightarrow [0,1]$ supported
in $B(0,s /2 )$ such that it is equal to 1 on $B(0,s /4)$. Assume that $| \nabla  \varphi_{s} |^2 , ( \Delta \varphi_{s} )^2
\leq D$, for some positive constant $D$ (which depends on $s$). Then, for $\ep>0$ small, we define
\begin{equation}\label{Ueab2}
U_{\alpha ,\beta}^{\ep , s} :=\, \varphi_s \,U_{\alpha , \beta }^{\ep} 
\end{equation}
where $U_{\alpha , \beta }^{\ep} $ is given by (\ref{Ueab}).
\begin{lemma}\label{lemma41}
$\lim_{\ep \rightarrow 0} \lambda^{+,\ep}  ( U_{\alpha ,\beta}^{\ep , s} )=1$ and 
$\lim_{\ep \rightarrow 0} E_{\alpha , \beta , q }^{\ep , +}  (  \lambda^{+,\ep} ( U_{\alpha ,\beta}^{\ep , s} )  U_{\alpha ,\beta}^{\ep , s} )= m_{\alpha , \beta}$. 
\end{lemma}
\begin{proof} Note that
$$ \lambda^{+,\ep}  ( u ) = \left(         \frac{  \int_{\re^n} \left(\ep^4   (\Delta u )^2  +
 \ep^2 \beta  | \nabla  u |^2  +\alpha u^2 \right)dx   }{        \int_{\re^n}    (u^+ )^{q+1} dx     }       
\right)^{\frac{1}{q-1}}
$$

We have that $\Delta  U_{\alpha ,\beta}^{\ep , s}  = \Delta (U_{\alpha , \beta }^{\ep} \,\varphi_s)  =
\varphi_s \,\Delta  U_{\alpha ,\beta}^{\ep }  + 2 \langle \nabla   U_{\alpha ,\beta}^{\ep } , \nabla \varphi_s \rangle
+  U_{\alpha ,\beta}^{\ep }  \Delta \varphi_s $.

Then 
$$
\left|\Delta  U_{\alpha ,\beta}^{\ep , s} \right|^2 
\begin{array}[t]{ll}\displaystyle =  \left| \varphi_s \Delta  U_{\alpha ,\beta}^{\ep }  \right|^2 + \left| 2 \langle \nabla   U_{\alpha ,\beta}^{\ep }  , \nabla \varphi_s \rangle
+   U_{\alpha ,\beta}^{\ep }  \Delta \varphi_s \right|^2 \medskip\\
\displaystyle\qquad  + 
2 \varphi_s \Delta   U_{\alpha ,\beta}^{\ep } 
 \left( 2 \langle \nabla  U_{\alpha ,\beta}^{\ep }  , \nabla \varphi_s \rangle
+  U_{\alpha ,\beta}^{\ep }  \Delta \varphi_s \right).
\end{array} 
$$

Note that for any $\delta >0$ we have that 
$$ |  2 \varphi_s \Delta  U_{\alpha ,\beta}^{\ep } 
(  2 \langle \nabla   U_{\alpha ,\beta}^{\ep }  , \nabla \varphi_s \rangle
+  U_{\alpha ,\beta}^{\ep }  \Delta \varphi_s )  |  \leq \delta \left| \varphi_s^2 \Delta   U_{\alpha ,\beta}^{\ep } 
\right|^2 + \delta^{-1} 
\left| 2 \langle \nabla  U_{\alpha ,\beta}^{\ep }  , \nabla \varphi_s \rangle
+  U_{\alpha ,\beta}^{\ep }  \Delta \varphi_s \right|^2 .
$$

It follows that we can pick $\delta$ small and write
$$
\begin{array}[t]{ll}
\displaystyle  \varphi^2_s (1 - \delta )  \left( \ep^4  (\Delta U_{\alpha , \beta }^{\ep} )^2 + O(\ep^4 )
(| \nabla U_{\alpha , \beta }^{\ep} |^2 + (U_{\alpha , \beta }^{\ep} )^2) \right) \medskip\\
\displaystyle 
\quad \leq
\ep^4 ( \Delta  U_{\alpha ,\beta}^{\ep , s} )^2 \medskip\\
\displaystyle 
\quad \leq \varphi^2_s  (1+ \delta )  ( \ep^4  (\Delta U_{\alpha , \beta }^{\ep} )^2 + O(\ep^4 )
(| \nabla U_{\alpha , \beta }^{\ep} |^2 + (U_{\alpha , \beta }^{\ep} )^2 ) .
\end{array}
$$

Similarly 
$$ 
| \nabla U_{\alpha ,\beta}^{\ep , s} |^2 = |  \varphi_s \nabla U_{\alpha ,\beta}^{\ep } 
+ U_{\alpha ,\beta}^{\ep} \nabla
\varphi_s |^2 \leq (1+ \delta ) \varphi_s^2 | \nabla U_{\alpha ,\beta}^{\ep } |^2 + (1+\delta^{-1} ) 
| \nabla \varphi_s |^2  ( U_{\alpha ,\beta}^{\ep } )^2 
$$
and 
$$ 
| \nabla U_{\alpha ,\beta}^{\ep , s} |^2   \geq (1- \delta ) \varphi_s^2 | \nabla U_{\alpha ,\beta}^{\ep } |^2 + (1-\delta^{-1} ) 
| \nabla \varphi_s |^2  ( U_{\alpha ,\beta}^{\ep } )^2 .
$$

It follows that
$$
\begin{array}[t]{lll}
\displaystyle \int_{\re^n} \big( \ep^4( \Delta  U_{\alpha ,\beta}^{\ep , s} )^2  +\ep^2 \beta | \nabla U_{\alpha ,\beta}^{\ep , s} |^2 +
\alpha (  U_{\alpha ,\beta}^{\ep , s} )^2 \big)\,dx \medskip\\
\displaystyle\quad \leq  \int_{\re^n} \big(\ep^4  (1+\delta ) \varphi_s^2  ( \Delta  U_{\alpha ,\beta}^{\ep } )^2  
+\ep^2 (1+\delta ) \beta  \varphi_s^2  (1+o(\ep^2 ) )  | \nabla U_{\alpha ,\beta}^{\ep } |^2 +
\alpha \varphi_s^2 (1+o(\ep^2 )) (  U_{\alpha ,\beta}^{\ep } )^2 \big) \,dx.
\end{array}
$$
Therefore for any $\delta >0$ we can choose $\ep_0 >0$ such that for every $\ep \in (0,\ep_0)$ we have that 
\begin{equation} 
\begin{array}[t]{lll}
\displaystyle \ep^{-n}\int_{\re^n} \big(\ep^4 ( \Delta  U_{\alpha ,\beta}^{\ep , s} )^2  +\ep^2 \beta | \nabla U_{\alpha ,\beta}^{\ep , s} |^2 + \alpha (  U_{\alpha ,\beta}^{\ep , s} )^2\big)\, dx \medskip\\
\displaystyle\qquad \leq   \ep^{-n}   (1+2\delta )\int_{\re^n} \big(\ep^4 (\Delta  U_{\alpha ,\beta}^{\ep } )^2  +\ep^2 \beta | \nabla U_{\alpha ,\beta}^{\ep } |^2 +
\alpha (  U_{\alpha ,\beta}^{\ep } )^2 \big)\,dx\medskip\\
\displaystyle \label{4.2}
\qquad= (1+2\delta) 
\int_{\re^n} \big( (\Delta  U_{\alpha ,\beta} )^2  +\beta | \nabla U_{\alpha ,\beta} |^2 +
\alpha (  U_{\alpha ,\beta} )^2 \big)\, dx .
\end{array}
\end{equation}

In the same way we can assume that for the same values of $\delta , \ep_0$, and any $\ep \in (0,\ep_0 )$ to get 
\begin{equation}
\begin{array}[t]{ll}
\displaystyle \ep^{-n} \int_{\re^n} \big(\ep^4 (\Delta  U_{\alpha ,\beta}^{\ep , s} )^2  +\ep^2 \beta | \nabla U_{\alpha ,\beta}^{\ep , s} |^2 + \alpha (  U_{\alpha ,\beta}^{\ep , s} )^2\big)\, dx \medskip\\
\label{4.3}
\displaystyle \qquad \geq
(1-2\delta) 
\int_{\re^n}  \big(\Delta ( U_{\alpha ,\beta} )^2  +\beta | \nabla U_{\alpha ,\beta} |^2 +
\alpha (  U_{\alpha ,\beta} )^2 \big)\, dx .\end{array}
\end{equation}

We can also see from (\ref{Ueab}) and (\ref{Ueab2}), that
\begin{equation}\label{4.4}
\lim_{\ep \rightarrow 0} \ \ep^{-n}  \int_{\re^{n}} (U_{\alpha ,\beta}^{\ep , s})^{q+1} dx = \int_{\re^{n}} 
 (  U_{\alpha ,\beta} )^{q+1}.
\end{equation}

Since $U_{\alpha , \beta} \in \mathcal{N}({\alpha , \beta , q})$ we have that  $\lim_{\ep \rightarrow 0} \lambda^{+,\ep}  (U_{\alpha ,\beta}^{\ep , s})=1$.

Therefore we are only left to prove that  $\lim_{\ep \rightarrow 0} E_{\alpha , \beta , q }^{\ep , +}   (U_{\alpha ,\beta}^{\ep , s})= m_{\alpha , \beta}$.
But this also follows directly from (\ref{4.2}), (\ref{4.3}), (\ref{4.4}).
\end{proof}

\section{The variational setting on a general Riemannian manifold}\label{section5}

Consider a closed $n$-dimensional Riemannian manifold $(M ,g)$. Let $\alpha , \beta$ be positive constants,
$\beta^2 \geq 4\alpha$. As defined in the previous section, for
Euclidean space, we consider
$J^{g,\ep}_{\alpha , \beta } : H^2 (M) \rightarrow \re$ where 
$$
J^{g, \ep}_{\alpha , \beta } (u) = \frac{\ep^{-n} }{2}  \int_{M} \left( \ep^4 ( \Delta_g u )^2  +
 \beta \ep^2   | \nabla  u |^2  +\alpha u^2 \right)dv_g.
$$

For $q>1$ such that if $n\geq 5$,  $q<\frac{n+4}{n-4}$,  let
$$
E^{g,\ep ,+}_{\alpha , \beta , q } (u) =\frac{ \ep^{-n} }{2} \int_{M} \left(\ep^4  (\Delta_g u )^2  +
 \ep^2 \beta  | \nabla  u |^2  +\alpha u^2 \right)dv_g - \frac{\ep^{-n} }{q+1} \int_{M}  (u^+ )^{q+1} dv_g.
$$

Let 
$$
\mathcal{N}^+_g (\alpha , \beta , q, \ep)=\bigg \{ u \in H^2 (M ) \setminus \{ 0 \}  : 2 J^{\ep}_{\alpha , \beta} (u) = 
\ep^{-n}  \int_{M}     (u^+ )^{q+1} dv_g  \bigg\},
$$ 
%
$$
m^{g,+}_{\alpha , \beta , \ep} = \inf_{u \in \mathcal{N}^+_g (\alpha , \beta , q ,\ep )} E^{\ep , +}_{\alpha , \beta , q} (u)  =
 \inf_{u \in \mathcal{N}^+_g(\alpha , \beta , q, \ep )}  \ep^{-n}  \frac{q-1}{2(q+1)} \| u^+  \|^{q+1}_{L^{q+1}(M)} \geq 0  .  
$$

Note as before that for any $u \in H^2 (M )$ such that $u^+ \neq 0$ there is a unique positive number
$\lambda^{g,\ep ,+} (u)$ such that $\lambda^{g,\ep ,+} (u)  .   u \in \mathcal{N}^+_g (\alpha , \beta , q, \ep)$.

Also define $Y^{g,\ep ,+}_{\alpha , \beta , q } : \{ u \in H^2 (M) : u^+ \neq 0 \} \rightarrow \re$, 
$$
Y^{g,\ep ,+}_{\alpha , \beta , q } (u)
\begin{array}[t]{ll} \displaystyle =  \frac{ \frac{ \ep^{-n} }{2}  \int_{M}\left(\ep^4  ( \Delta_g u )^2  +
 \ep^2 \beta  | \nabla  u |^2  +\alpha u^2 \right)dv_g  }{  \left( \ep^{-n} \int_{M}  (u^+ )^{q+1} dv_g 
\right)^{\frac{2}{q+1}} }\medskip\\ \displaystyle = \ep^{-n\frac{q-1}{q+1}} \frac{ \frac{ 1 }{2}  \int_{M} \left(\ep^4  ( \Delta_g u )^2  +
 \ep^2 \beta  | \nabla  u |^2  +\alpha u^2 \right)dv_g  }{  \left(  \int_{M}  (u^+ )^{q+1} dv_g 
\right)^{\frac{2}{q+1}} }.\end{array}
$$
Note that if $u \in \mathcal{N}^+_g(\alpha , \beta , q, \ep )$ then $Y^{g,\ep ,+}_{\alpha , \beta , q } (u)  =(1/2) 
\ep^{-n\frac{q-1}{q+1}} \| u^+  \|^{q-1}_{L^{q+1}(M)}$. Then

\begin{equation} m^{g,+}_{\alpha , \beta , \ep}  \begin{array}[t]{lll}
\displaystyle =  \inf_{u \in \mathcal{N}_g^+(\alpha , \beta , q, \ep )}  \ep^{-n}  \frac{q-1}{2(q+1)} \| u^+  \|_{L^{q+1}(M)}^{q+1}\medskip\\ 
\displaystyle =  \frac{q-1}{2(q+1)} \bigg(   \inf_{u \in \mathcal{N}_g^+(\alpha , \beta , q, \ep )}   \ep^{-n\frac{q-1}{q+1}}   \| u^+  \|_{L^{q+1}(M)}^{q-1} 
\bigg)^{\frac{q+1}{q-1} }\medskip\\ 
\displaystyle= \frac{q-1}{q+1} 2^{\frac{q+1}{q-1}}  \left(   \inf_{u \in \mathcal{N}^+(\alpha , \beta , q, \ep )}  Y^{g,\ep ,+}_{\alpha , \beta , q }
 (u) \right)^{\frac{q+1}{q-1}} 
 \medskip\\ 
\displaystyle \label{4.1}
= \frac{q-1}{q+1} 2^{\frac{2}{q-1}}  \left(   \inf_{u \in  \{ u \in H^2 (M) : u^+ \neq 0 \}  }  Y^{g,\ep ,+}_{\alpha , \beta , q }
 (u) \right)^{\frac{q+1}{q-1}}.\end{array}
\end{equation}
\begin{remark}\label{remark4}
For any $x\in M$ identify the geodesic ball of some small radius $r$ with the ball of radius $r$ centered at
0 in $\re^n$ using the exponential map $exp_x$. Let $\overline{U}_{\alpha ,\beta ,x}^{\ep , s}  = exp_x \circ U_{\alpha ,\beta}^{\ep , s}$.
Note that for any $\delta >0$ we can find $r>0$ such that on the ball $B_g(0,r)$  is $\delta$-close to
the Euclidean metric in the $C^1$-topology. It follows that the volume elements, gradients, and 
Laplacian corresponding to $g$ and to the Euclidean metric are $\delta$-close. 
\end{remark}

Since for $\ep$ small enough
we can get that $ U_{\alpha ,\beta}^{\ep , s}$ is concentrated in any small ball it follows 
directly from the previous remark and  Lemma \ref{lemma41} that:
\begin{proposition} \label{propo53}We have 
$$\lim_{\ep \rightarrow 0} \lambda^{g,\ep ,+}  (\overline{U}_{\alpha ,\beta ,x}^{\ep , s } )  
=1\quad\mbox{and}\quad 
\lim_{\ep \rightarrow 0} E^{\ep ,+}_{\alpha , \beta , q}  (  \lambda^{g,\ep ,+}  (\overline{U}_{\alpha ,\beta, x}^{\ep , s} )
 \overline{U}_{\alpha ,\beta , x }^{\ep , s} )= m_{\alpha , \beta}.$$ 
\end{proposition}

\begin{definition}\label{definition4.4} With the notations from the previous remark 
define ${\bf i} : M \rightarrow \mathcal{N}(\alpha , \beta , q, \ep)$ by $${\bf i} (x):= 
 \lambda^{g,\ep , +}  (\overline{U}_{\alpha ,\beta, x}^{\ep , s} )
 \overline{U}_{\alpha ,\beta , x }^{\ep , s}.$$ 
\end{definition}

We prove
\begin{theorem}\label{theorem4}For any $n$-dimensional closed Riemannian manifold $(M ,g)$ we have that 
$$\lim_{\ep \rightarrow 0} m^{g,+}_{\alpha , \beta , \ep}  = m_{\alpha , \beta}.$$
\end{theorem}
\begin{proof} 
Since for any $x\in M$, 
$m^{g,+}_{\alpha , \beta , \ep}  \leq  E^{\ep ,+}_{\alpha , \beta , q} ({\bf i} (x) )$
it follows from the previous proposition that $\limsup_{\ep \rightarrow 0} 
m^{g,+}_{\alpha , \beta , \ep} \leq  m_{\alpha , \beta}$.

To complete the proof of the theorem we   prove that 
$\liminf_{\ep \rightarrow 0} m^{q,+}_{\alpha , \beta , \ep}
 \geq  m_{\alpha , \beta}$.

For any $\delta \in (0,1)$ we can pick normal neighborhoods $U_1 , ..., U_L$ such
that in normal coordinates $g$ is $C^1$-close to the Euclidean metric. Namely
$$
(1+\delta )^{-1}  \leq g_{ii} \leq (1+\delta) 
$$
$$ 
|g_{ij} | \leq \delta , \  \  \  i\neq j
$$
$$ 
|g_{ij,k} | \leq \delta  
$$
and
$$
(1+\delta )^{-1} dx  \leq dv_g \leq (1+\delta)  dx 
$$
Let $\varphi_k^2$, $k=1, \dots , L$ be a partition of unity subordinate to $U_1 , \dots , U_L$.
Let $D$ be an upper bound for $| \nabla \varphi_k |^2 , ( \Delta \varphi_k )^2$, $k=1, \dots , L$.

For any $u\in H^2 (M)$ we have
$$
\| \ep^{\frac{-n}{q+1}}  u^+ \|^2_{L^{q+1}(M)}  \begin{array}[t]{ll} \displaystyle \leq  \|  \ep^{\frac{-n}{q+1}} u \|^2_{L^{q+1}(M)}\medskip\\
\displaystyle 
 = \|  \ep^{\frac{-2n}{q+1}} u^2 \|_{L^{\frac{q+1}{2}}(M)} \medskip\\
\displaystyle   
= \Big\| \sum_k   \ep^{\frac{-2n}{q+1}}  \varphi_k^2 u^2 \Big\|_{L^{\frac{q+1}{2}}(M)} \medskip\\
\displaystyle  \leq \sum_k   \left( \ep^{-n}  \int_{U_k} | \varphi_k u  |^{q+1} dv_g \right)^{\frac{2}{q+1}} \medskip\\
\displaystyle  
\leq (1+\delta )^{\frac{2}{q+1}} \sum_k   \left( \ep^{-n}  \int_{U_k} | \varphi_k u  |^{q+1} dx \right)^{\frac{2}{q+1}} .\end{array}
$$

Using (\ref{3.7}) we get 
$$ 
\left( \ep^{-n}  \int_{U_k} | \varphi_k u  |^{q+1} dx \right)^{\frac{2}{q+1}}    \leq
\left( {\bf Y}^{\ep}_{\alpha , \beta , q} \right)^{-1}  J^{\ep}_{\alpha , \beta } (\varphi_k u ). 
$$
But
$$ 
J^{\ep}_{\alpha , \beta } (\varphi_k u ) \begin{array}[t]{ll}
\displaystyle =\frac{\ep^{-n}}{2} \int_{U_k} \big(
\ep^4 ( \Delta (\varphi_k u )  )^2 + \ep^2  \alpha | \nabla (\varphi_k u ) |^2 
+\beta (\varphi_k u )^2 \big)\,dx  \medskip\\
\displaystyle \leq 
(1+\delta )^2 \frac{\ep^{-n}}{2} \int_{U_k} \big(\ep^4 ( \Delta_g (\varphi_k u )  )^2 +
\ep^2 \alpha | \nabla (\varphi_k u ) |^2 
+\beta (\varphi_k u )^2 \big)\,dv_g. \end{array}
$$

Arguing as in the proof of Lemma \ref{lemma41} we can now see that for $\ep$ small enough
$$ 
J^{\ep}_{\alpha , \beta } (\varphi_k u ) \leq (1+\delta ) (1+\delta )^2 
\frac{\ep^{-n}}{2} \int_{U_k} \ep^4  \varphi_k^2 ( \Delta_g u  )^2 +
\ep^2 \alpha \varphi_k^2  | \nabla u |^2 
+\beta \varphi_k^2 u^2 dv_g .
$$

Then 
$$
\begin{array}[t]{ll}\displaystyle 
 \left( \ep^{-n} \int_M  ( u^+ )^{q+1} dv_g \right)^{\frac{2}{q+1}} 
\medskip\\
\displaystyle \qquad\leq  (1+\delta )^{\frac{2}{q+1}} \left( {\bf Y}^{\ep}_{\alpha , \beta , q} \right)^{-1}  
(1+\delta )^3  \sum_k  
\frac{\ep^{-n}}{2} \int_{U_k} \big(\ep^4  \varphi_k^2 ( \Delta_g u  )^2 +
\ep^2 \alpha \varphi_k^2  | \nabla u |^2 
+\beta \varphi_k^2 u^2 \big)\, dv_g \medskip\\
\displaystyle \qquad = (1+\delta )^{\frac{2}{q+1}} \left( {\bf Y}^{\ep}_{\alpha , \beta , q} \right)^{-1}  (1+\delta )^3 J^{\ep}_{\alpha , \beta } (u).
\end{array}
$$
Therefore
$$ 
Y^{g, \ep ,+}_{\alpha , \beta ,q} (u) \geq  \frac{1}{ (1+\delta )^4}  
{\bf Y}^{\ep}_{\alpha , \beta , q} .
$$
Then by (\ref{3.6}) and (\ref{4.1}) it follows that 
$$
m^{g,+}_{\alpha , \beta , \ep}  \geq  \frac{1}{ (1+\delta )^4}  m_{\alpha , \beta}.
$$
Since this could be done for any $\delta >0$ it follows that 
$$
\liminf_{\ep \rightarrow 0} m^{g,+}_{\alpha , \beta , \ep} 
 \geq  m_{\alpha , \beta}.
$$
\end{proof}

Now define
$$
J_{\ep}^g (u):=\frac{ \ep^{-n}}{2}  \int_M \ep^4 u  \, P_{\geph} u \,dv_g,
$$
where $P_{\geph}$ is given in (\ref{Pgep}).
Observe that $\mathcal{N}^g_{\ep} $ in (\ref{N_ep}) is given by 
$$
\mathcal{N}^g_{\ep} =\bigg \{ u \in H^2 (M) :  \int_M \big(\ep^4 u  P_{\geph} u \big)\,dv_g = \int_M  (u^+ )^{q+1}  dv_g\bigg\},
$$
and $J^{g,+}_{\ep , q}(u) $ in (\ref{J_ep}) is given by 
$$
J^{g,+}_{\ep , q} (u):= \frac{\ep^{-n}}{2}  \int_M  \ep^4 u \, P_{\geph} u\, dv_g - \frac{\ep^{-n}}{q+1}
\int_M  (u^+ )^{q+1} \, dv_g.
$$
We now put
$$
{\bf m}_{\ep} := \inf_{u \in \mathcal{N}^g_{\ep} } J^{g,+}_{\ep , q} (u) .
$$
It follows from Remark \ref{remark1} and Theorem \ref{theorem4} that 
\begin{theorem}\label{theorem5} 
It is verified that 
$$\lim_{\ep \rightarrow 0} {\bf m}_{\ep} = 
m_{a_{N,m}, b_{N,m}}.$$
\end{theorem}

\section{Concentration for functions on $\Sigma_{\varepsilon, {\bf m}_{\ep} + \delta}$}\label{section6}

For  any $d >0$ we let $\Sigma_{\varepsilon, d} = \{u\in \mathcal{N}^g_{\varepsilon}  : J^{g,+}_{\varepsilon, q} (u) <  d \}$.
In this section we   show that for $\ep > 0$, $\delta >0$ small enough, the functions in $\Sigma_{\varepsilon , {\bf m}_{\ep} + \delta}$ concentrate on a small ball. Namely, we   prove:
\begin{theorem}\label{theorem5.1} 
Fix $\eta <1$ and $r>0$. There exist $\ep_0 , \delta _0 >0$ such that for any $\ep \in (0,\ep_0 )$, $\delta \in (0,\delta_0 )$ and any $u\in 
\Sigma_{\varepsilon , {\bf m}_{\ep} + \delta}$ there exists $x\in M$ such that
$$
\int_{B_g(x,r)}  ( u^+ )^{q+1} dv_g  \geq \eta \int_M  ( u^+ )^{q+1} dv_g .
$$
\end{theorem}

Let $\alpha = a_{N,m}$, $\beta ={b_N,m}$. 
Consider as in the previous section
$$J^{g,\ep}_{\alpha , \beta } (u) = \frac{\ep^{-n} }{2}  \int_{M} \big( \ep^4 ( \Delta_g u )^2  + \beta \ep^2   | \nabla  u |^2  +\alpha u^2\big) \,dv_g,$$
$$E^{g,\ep ,+}_{\alpha , \beta , q } (u) =\frac{ \ep^{-n} }{2} \int_{M} \big(\ep^4  ( \Delta_g u )^2  +
 \ep^2 \beta  | \nabla  u |^2  +\alpha u^2\big)\, dv_g - \frac{\ep^{-n} }{q+1} \int_{M}  (u^+ )^{q+1} dv_g,$$
$$\mathcal{N}_g^+ (\alpha , \beta , q, \ep)= \bigg\{ u \in H^2 (M ) \setminus \{ 0 \}  : 2 J^{g,\ep}_{\alpha , \beta} (u) = 
\ep^{-n}  \int_{M}     (u^+ )^{q+1} dv_g \bigg \},
$$
\begin{equation}\label{equivalencias}
m^{g,+}_{\alpha , \beta , \ep} \begin{array}[t]{ll}
\displaystyle = \inf_{u \in \mathcal{N}^+ (\alpha , \beta , q ,\ep )} E^{\ep , +}_{\alpha , \beta , q} (u) \medskip \\
\displaystyle=  \inf_{u \in \mathcal{N}^+ (\alpha , \beta , q ,\ep )} J^{g,\ep}_{\alpha , \beta } (u) \medskip\\
\displaystyle = \inf_{u \in \mathcal{N}^+ (\alpha , \beta , q ,\ep )} \frac{q-1}{2(q+1)} \ep^{-n} 
 \int_{M}     (u^+ )^{q+1} dv_g .\end{array}
\end{equation}

Let
$$\Sigma^{\alpha , \beta}_{\ep , d} = \{ u \in \mathcal{N}_g^+ (\alpha , \beta , q, \ep) :
E^{g,\ep ,+}_{\alpha , \beta , q } (u) \leq d \}.
$$

In the rest of the paper, we will use the following notation:  $B_g(x,r)$ is the ball in $M$ centered at $x$ with radius r with respect to the distance induced by the metric $g$. 

\medskip

Let $r_0=r_0 (M,g)$ be such that for any $0<r<r_0$ and any 
$x\in M$, $B_g (x,r_0)$ is strongly convex. 

\medskip

Taking into account Remark \ref{remark1} and Theorem \ref{theorem5} we see that 
Theorem \ref{theorem5.1} follows from

\begin{theorem}\label{theorem5.2} Fix $r < r_0 (M,g)$. For any $\eta <1$ there exist $\varepsilon_0  , \delta_0 >0$ such that for any $\varepsilon \in (0 , \varepsilon_0 )$, $\delta
\in (0,\delta_0 )$ and $u\in \Sigma^{\alpha , \beta}_{\varepsilon , m^{g,+}_{\alpha , \beta , \ep} + \delta}$ there exists  $x \in M$ such that
$\int_{B_g(x,r)}  (u^+ )^{q+1} \geq \eta \int_M ( u ^+ )^{q+1} $.
\end{theorem}

For any function $u \in H^2 (M)$ such that $u^+ \neq 0$  we let as before
$$
\lambda^{g,\ep ,+} (u) = { \left(   \frac{\int_M\big( \ep^4  ( \Delta_g u )^2  +
 \ep^2 \beta  | \nabla  u |^2  +\alpha u^2\big)\, dv_g }{\int_M  (u^+ )^{q+1} dv_g } \right) }^{\frac{1}{{q-1}}} ,
$$
so that $\lambda^{g,\ep ,+} (u)\, u \in \mathcal{N}^+_g (\alpha , \beta , q, \ep)$

\begin{proposition} Let $u=u_1 + u_2  \in \mathcal{N}^+_g (\alpha , \beta , q, \ep)$, where $u_1$ and $u_2$ 
have disjoint supports.
For $i=1,2$ let $\ep^{-n} \int_M  ( u_i ^+  )^{q+1} dv_g =a_i   >0$  and 
$\ep^{-n} \int_M  \big(  \ep^4  ( \Delta_g u )^2  +
 \ep^2 \beta  | \nabla  u |^2  +\alpha u^2 \big)\, dv_g  =
 b_i >0$.
Note that then $a_1 + a_2 = b_1 + b_2$ (since  $u \in  \mathcal{N}_g^+ (\alpha , \beta , q, \ep)$) .
Then
$$ 
E^{\ep , +}_{\alpha , \beta , q} (u) \geq m^{g,+}_{\alpha , \beta , \ep} \left(  {\left( \frac{a_1}{b_1} \right) }^{\frac{2}{q-1}} + {\left(  \frac{a_2 }{b_2} \right) }^{\frac{2}{q-1}}  \right) .
$$
\end{proposition}

\begin{proof} Note that $\lambda^{g, \ep ,+}  (u_i )= \left( \frac{b_i }{a_i} \right)^{\frac{1}{q-1}}.$ 
Then 
$$
m^{g,+}_{\alpha , \beta , \ep} \leq  E^{\ep , +}_{\alpha , \beta , q} \left(  \left( \frac{b_i }{a_i} 
\right)^{\frac{1}{q-1}}
u_i \right) = 
\left( \frac{b_i}{a_i} \right)^{\frac{2}{q-1}} \frac{q-1}{2(q+1)}   \  b_i ,
$$
or equivalently, 
$$
b_i \geq  \frac{2(q+1)}{q-1} \left( \frac{a_i}{b_i} \right)^{\frac{2}{q-1}} m^{g,+}_{\alpha , \beta , \ep}.
$$
Then
$$
E^{\ep , +}_{\alpha , \beta , q} (u) = \frac{q-1}{2(q+1)}  (b_1 + b_2 ) \geq   
m^{g,+}_{\alpha , \beta , \ep} \left(
{\left( \frac{a_1}{b_1} \right)}^{\frac{2}{q-1}}    +   {\left( \frac{a_2}{b_2} \right)}^{\frac{2}{q-1}} \right).
$$
\end{proof}

Now consider the following elementary result:
\begin{lemma} Let $r>0$ and $x_1 , x_2 , y_1 ,y_2$ be positive numbers such that $x_1 + x_2 = y_1 + y_2 =r$. Then
$$ 
\varphi (x_1 , x_2 , y_1 , y_2 ) = {\left(  \frac{x_1}{y_1} \right) }^{\frac{2}{q-1}} + {\left( \frac{x_2}{y_2}    \right)}^{\frac{2}{q-1}} >1  .
$$
For any $\delta \in (0,1)$ consider the set 
$$
A_{\delta} :=\left \{ (x_1 , x_2 , y_1 , y_2 ) \in (\re_{>0} )^4  : x_1 + x_2 = y_1 + y_2 =r, \  x_1 , x_2 \geq \delta r \right\}
$$
and let
$\Psi (\delta ) := \inf_{A_{\delta}} \varphi .$ Then  $\Psi : (0,1) \rightarrow (1 ,\infty )$ is a continuous function, and $\Psi$ is
independent of $r$.
\end{lemma}

We can deduce
\begin{corollary}\label{coro65} Let $u = u_1 + u_2 \in  \mathcal{N}_g^+ (\alpha , \beta , q, \ep)$, where the functions $u_1$, $u_2$ 
have disjoint supports. Moreover assume that $\ep^{-n} \int_M  ( u_i ^+  )^{q+1} dv_g > \delta$
for some $\delta >0$, $i=1,2$.
Then $$ E^{\ep , +}_{\alpha , \beta , q} (u)  \geq \Psi (\delta )  \  m^{g,+}_{\alpha , \beta , \ep} .$$
\end{corollary}
 
Fix $r < r_0 (M,g)$.

Fix a positive integer $l$ and another positive integer $i$, $1<i < l$. Consider   bump functions $\varphi_{i,l}$ which is equal to
1 in $\Big[0,\frac{(i-1) }{l} r \Big]$ and 0 in $\Big[\frac{i}{l} r , \infty\Big)$, and  $\sigma_{i,l}$ which is 1 in $\Big[\frac{(i+1)}{l} r , \infty\Big )$ and
0 in $\Big[0 , \frac{i}{l} r \Big]$.

\begin{remark} We can assume that $ |(  \varphi_{i,l}  ) '| ,  | ( \sigma_{i,l} ) ' | \leq
3l/r$, $|  ( \varphi_{i,l}  ) '' | , | ( \sigma_{i,l} ) '' | \leq 9 (l/r)^2$.
\end{remark}

For any  $x\in M$, consider the   function $\varphi_{i,l}^{r,x} :  M \rightarrow [0,1]$ which
is 0 away from $B_g(x,r) \subset M$ and in $ B_g (x,r)$ (identified with the $r$-ball in $\re^n$) is given by
$\varphi_{i,l}  ( d(0,-))$. Define  $\sigma_{i,l}^{r,x}$ in a similar way. Now for any function $u \in \mathcal{N}^g_{\varepsilon}$ define
$u_{1,i,l}^{r,x} = \varphi_{i,l}^{r,x} \ u$ and $u_{2,i,l}^{r,x} = \sigma_{i,l}^{r,x} \ u$.
Note that since $u^+ \neq 0$  we have that $ ( u_{1,i,l}^{r,x}  + u_{2,i,l}^{r,x} )^+ \neq 0$.
Note also that  these two functions have disjoint supports.

Let
$$
\overline{u}   := \overline{u}_{i,l}^{r,x} =     \lambda^{q,\ep ,+} (  u_{1,i,l}^{r,x}  + u_{2,i,l}^{r,x}   ) \,(  u_{1,i,l}^{r,x}  + u_{2,i,l}^{r,x}   ) .
$$

\begin{lemma}\label{lema67} For the closed Riemannian manifold $(M,g)$ fix $0< r <  r_0 (M,g) $ and  $\delta_0 >0$. There exists $\varepsilon_0 >0$ and $l \in \mathbb{ N}$
 such that for any $\varepsilon \in (0, \varepsilon_0 )$ we have that
 for any $u \in \Sigma^{\alpha , \beta}_{\varepsilon , m^{g,+}_{\alpha , \beta , \ep}  + \delta_0  }$ and any $x \in M$ there exist $i$, $1<i<l$ such that $ \overline{u}
\in \Sigma^{\alpha , \beta}_{\varepsilon , m^{g,+}_{\alpha , \beta , \ep} +  2  \delta_0 } $
\end{lemma}
\begin{proof}
Assume from the beginning that $\varepsilon$ is small enough so that $m^{g,+}_{\alpha , \beta , \ep}$ is close to $m_{\alpha , \beta}$ (using Theorem \ref{theorem4}).
Pick $\eta \in (0,1)$, close to 1,  such that (for any such $\ep$)

\begin{equation}\label{5.1}
(m^{g,+}_{\alpha , \beta , \ep}  + \delta_0  )  {\left( \frac{9 - 8 \eta }{\eta} \right) }^{\frac{2}{q-1}} (9 - 8 \eta ) < m^{g,+}_{\alpha , \beta , \ep}  +  2 \delta_0.
\end{equation}

We   assume that  the positive integer $l$ is  large enough so that $8/l < 1-\eta .$

Consider  any $\varepsilon$ small and 
$u \in \Sigma_{\varepsilon , m^+_{\alpha , \beta , \ep}   + \delta_0 }$. Pick any $x\in M$.
Assume $l$ is even and divide $B_g(x,r)$ into $l/2$ annular regions of radius $(2/l)r$. 
Call $A_j$, $j=1,\ldots,l/2$ each of these
regions. 

For any nonegative integrable function $f$ on $B(x,r)$ we have
$$
\int_{B_g(x,r)} f dv_g = \sum_{j=1}^{l/2} \int_{A_j} f dv_g .
$$

Then $\# \{ j: \int_{A_j} f dv_g \geq (8/l) \int_{B_g(x,r)} f dv_g  \}  \leq  l/8$. Applying this observation
to $(u^+ ) ^{q+1} ,  (\Delta_g u )^2 $ and $ | \nabla u |^2 $ we see that  there exists
$j>1$ such that
\begin{equation}\label{1}
\int_{A_j }  (\Delta_g u )^2  dv_g \leq \frac{8}{l}  \int_M   ( \Delta_g u )^2 dv_g <  (1- \eta )
 \int_M   (\Delta_g u)^2 dv_g ,
\end{equation}

\begin{equation}\label{11}
\int_{A_j }  |  \nabla u |^2  dv_g \leq \frac{8}{l}  \int_M   |  \nabla u |^2 dv_g <  (1- \eta ) \int_M   |  \nabla u |^2 dv_g 
\end{equation}
and
\begin{equation}\label{111}
\int_{A_j }   (u^+ ) ^{q+1}  dv_g \leq \frac{8}{l}  \int_M    (u^+ ) ^{q+1} dv_g <  (1- \eta ) \int_M    (u^+ ) ^{q+1}     dv_g .
\end{equation}

Note that if $i=2j-1$ then outside of $A_j$ we have that 
$u^+ = ( u_{1,i,l}^{r,x} )^+   +  (u_{2,i,l}^{r,x} )^+ $. Therefore from (\ref{111}) we get 
\begin{equation}\label{5.4}
\big\| ( u_{1,i,l}^{r,x} )^+   +  (u_{2,i,l}^{r,x} )^+   \big \|_{q+1}^{q+1} > \eta \| u^+ \|_{q+1}^{q+1} .
\end{equation}

Note that  the gradients and Laplacians of $  u_{1,i,l}^{r,x} $ and $  u_{2,i,l}^{r,x} $ also have disjoint supports.
Also note that  $ | u_{1,i,l}^{r,x}  |  + |  u_{2,i,l}^{r,x}  | \leq  | u |$ and away from $A_j$, $u =  u_{1,i,l}^{r,x}  + u_{2,i,l}^{r,x} $.

We have
\begin{equation*}
\int_{A_j}  (\Delta_g (  u_{1,i,l}^{r,x}  + u_{2,i,l}^{r,x} ) )^2 dv_g =
\int_{A_j}  (\Delta_g  u_{1,i,l}^{r,x}    )^2 dv_g   +
\int_{A_j}  (\Delta_g  u_{2,i,l}^{r,x} )^2 dv_g
\end{equation*}
and
\begin{equation*}
\int_{A_j}  (\Delta_g (  u_{1,i,l}^{r,x} ) )^2 dv_g \begin{array}[t]{lll}
\displaystyle =\int_{A_j} \left(  \varphi_{i,l}^{x,r} \Delta_g u +
u \Delta_g  \varphi_{i,l}^{x,r}  +2 \langle \nabla  \varphi_{i,l}^{x,r} , \nabla u \rangle \right)^2
dv_g \medskip\\
\displaystyle 
\leq 3 \int_{A_j}  \big(( \varphi_{i,l}^{x,r} )^2 (\Delta_g u )^2 + u^2 (\Delta_g  \varphi_{i,l}^{x,r} )^2
+4 | \nabla u |^2 | \nabla  \varphi_{i,l}^{x,r} |^2  \big)\,dv_g\medskip\\
\displaystyle 
\leq 3 \int_{A_j} (\Delta_g u)^2 dv_g + 3 \Big(9 \frac{l^2}{r^2} \Big)^2 \int_{A_j} u^2 dv_g +12 
\Big(3\frac{l}{r}\Big)^2 
\int_{A_j} | \nabla u  |^2 dv_g \end{array}
\end{equation*}

A similar inequality is obtained for $u_{2,i,l}^{r,x}$ and we have:
\begin{equation*}
\int_{A_j}  (\Delta_g (  u_{1,i,l}^{r,x}  + u_{2,i,l}^{r,x} ) )^2 dv_g  \leq 
6 \int_{A_j} (\Delta_g u)^2 dv_g + 6 \Big(9 \frac{l^2}{r^2} \Big)^2
 \int_{A_j} u^2 dv_g +24 \Big(3\frac{l}{r}\Big)^2
\int_{A_j} | \nabla u  |^2 dv_g 
\end{equation*}

Using (\ref{1}) we obtain 
\begin{equation*}\label{equ511}
\begin{array}[t]{ll}
\displaystyle \int_{M}  (\Delta_g (  u_{1,i,l}^{r,x}  + u_{2,i,l}^{r,x} ) )^2 dv_g 
\medskip\\
\displaystyle =\int_{A_j}  (\Delta_g (  u_{1,i,l}^{r,x}  + u_{2,i,l}^{r,x} ) )^2 dv_g + \int_{M\setminus A_j}  (\Delta_g (  u_{1,i,l}^{r,x}  + u_{2,i,l}^{r,x} ) )^2 dv_g 
\medskip\\
\displaystyle
 \leq 
(6-6\eta ) \int_{M} (\Delta u)^2 dv_g + \int_{M} (\Delta u)^2 dv_g
+
6 \Big(9 \frac{l^2}{r^2} \Big)^2  \int_{M} u^2 dv_g +24 \Big(3\frac{l}{r}\Big)^2  
\int_{M} | \nabla u  |^2 dv_g .
\end{array} 
\end{equation*}

We have
$$\int_{A_j} |   \nabla ( u_{1,i,l}^{r,x}  + u_{2,i,l}^{r,x}  ) |^2  \  dv_g  =  \int_{A_j} |  \nabla   u_{1,i,l}^{r,x}   |^2   \  dv_g \ +  \     \int_{A_j}    |  \nabla    u_{2,i,l}^{r,x} |^2
\ dv_g$$
and
$$  \int_{A_j}  |  \nabla   u_{1,i,l}^{r,x}   |^2  \ dv_g  \begin{array}[t]{lll}
 \displaystyle = \int_{A_j}  |   \varphi_{i,l}^{x,r}  \nabla u + u \nabla  \varphi_{i,l}^{x,r}  |^2  \ dv_g \medskip\\
 \displaystyle  \leq  \ 2 \int_{A_j} \left( (  \varphi_{i,l}^{x,r} )^2 |\nabla u |^2 +  u^2 |   \nabla  \varphi_{i,l}^{x,r}  |^2  \right)  dv_g\medskip\\
 \displaystyle \leq 2 \int_{A_j} \left(  | \nabla u |^2 +9 \frac{l^2}{r^2} u^2 \right)  dv_g .\end{array}$$

The same estimate can be carried out for $\int_{A_j} | \nabla  u_{2,i,l}^{r,x}  |^2   \ dv_g $. Hence
$$\int_{A_j} |   \nabla ( u_{1,i,l}^{r,x}  + u_{2,i,l}^{r,x}  ) |^2  \  dv_g  \leq 
4 \int_{A_j} \left(  | \nabla u |^2 +9 \frac{l^2}{r^2} u^2 \right)  \ dv_g .$$

Hence, we obtain  from (\ref{11})
$$\int_{M} |   \nabla ( u_{1,i,l}^{r,x}  + u_{2,i,l}^{r,x}  ) |^2  \  dv_g  
\begin{array}[t]{lll} \displaystyle =\int_{A_j} |   \nabla ( u_{1,i,l}^{r,x}  + u_{2,i,l}^{r,x}  ) |^2  \  dv_g  
+\int_{M\setminus A_j} |   \nabla ( u_{1,i,l}^{r,x}  + u_{2,i,l}^{r,x}  ) |^2  \  dv_g  \medskip\\
\displaystyle 
\leq (4-4 \eta ) \int_M | \nabla u |^2 dv_g + \int_{M\setminus A_j} |   \nabla u |^2 dv_g 
+ 36\frac{i^2}{r^2} \int_M u^2 dv_g \medskip\\
\displaystyle\leq (5-4 \eta ) \int_M | \nabla u |^2 dv_g
+ 36\frac{i^2 }{r^2 } \int_M u^2 dv_g .\end{array}$$

We have:
$$\begin{array}[t]{lll}\displaystyle  \int_M \big( \ep^4  \big( \Delta_g   ( u_{1,i,l}^{r,x}  + u_{2,i,l}^{r,x}  ) \big)^2 
+\beta  \ep^{2} | \nabla  ( u_{1,i,l}^{r,x}  + u_{2,i,l}^{r,x}  ) |^2
+\alpha   ( u_{1,i,l}^{r,x}  + u_{2,i,l}^{r,x}  )^2 \big)\, dv_g\medskip\\
\qquad \displaystyle\leq 
(7-6 \eta) \int_M \ep^4 ( \Delta_g u )^2 dv_g + O(\ep^4  ) \int_M\big( | \nabla u |^2 +u^2 \big)\, dv_g \medskip\\
\displaystyle\qquad\qquad+ (5-4 \eta ) \ep^2  \int_M | \nabla u |^2 dv_g
+ O(\ep^2 ) \int_M u^2 dv_g + \alpha
\int_M u^2 dv_g. \end{array}
$$

Therefore, for the values of  $\eta$ and $l$ we have already fixed, there is $\ep_0 >0$ such that for any $\ep \in (0,\ep_0 )$ 
$$ \begin{array}[t]{lll}\displaystyle \int_M \big( \ep^4  \big( \Delta_g   ( u_{1,i,l}^{r,x}  + u_{2,i,l}^{r,x}  )\big )^2 
+\beta  \ep^{2} | \nabla  ( u_{1,i,l}^{r,x}  + u_{2,i,l}^{r,x}  ) |^2
+\alpha   ( u_{1,i,l}^{r,x}  + u_{2,i,l}^{r,x}  )^2 \big)\, dv_g\medskip\\
\qquad \displaystyle\leq (9-8\eta) \int_M \big(\ep^4 (\Delta_g u )^2 + \beta  \ep^2 | \nabla u |^2 + \alpha u^2 \big) \,dv_g .\end{array} $$

It follows from the previous inequality, the formula for $\lambda^{g, \ep ,+}$  and (\ref{5.4}) that
$$\lambda^{g,\ep, +}   (  u_{1,i,l}^{r,x}  + u_{2,i,l}^{r,x}  ) \leq   { \left( \frac{9- 8 \eta }{\eta}  \right)  }^{\frac{1}{q-1}} . $$

Then, for $\overline{u} = \lambda^{g, \ep ,+} (   u_{1,i,l}^{r,x}  + u_{2,i,l}^{r,x}  ) (   u_{1,i,l}^{r,x}  + u_{2,i,l}^{r,x}  ) \in \mathcal{N}_g^+ (\alpha , \beta , q, \ep)$, we have
$$ J^{g, \varepsilon }_{\alpha , \beta } (\overline{u} ) = $$
$$
\begin{array}[t]{ll}
\displaystyle 
  \frac{ ( \lambda^{g,\ep ,+} )^2  ( u_{1,i,l}^{r,x}  + u_{2,i,l}^{r,x}  )} {2 \varepsilon^n}    
 \int_M   \big(  \ep^4  ( \Delta   ( u_{1,i,l}^{r,x}  + u_{2,i,l}^{r,x}  ) )^2 
+\beta  \ep^{2} | \nabla  ( u_{1,i,l}^{r,x}  + u_{2,i,l}^{r,x}  ) |^2
+\alpha   ( u_{1,i,l}^{r,x}  + u_{2,i,l}^{r,x}  )^2 \big)\, dv_g  \medskip\\
\displaystyle  \leq   { \left( \frac{9- 8 \eta }{\eta}  \right)  }^{\frac{2}{q-1}}    (9- 8\eta )
 J^{g, \varepsilon }_{\alpha , \beta } (u).\end{array} $$

So if $u \in \Sigma^{\alpha , \beta}_{\varepsilon , m^+_{\alpha , \beta , \ep}  + \delta_0  }$ 
we have that 
$$ E^{\varepsilon , +}_{\alpha , \beta , q} (\overline{u} )
\leq   { \left( \frac{9- 8 \eta }{\eta}  \right)  }^{\frac{2}{q-1}}    (9- 8\eta )
E^{\varepsilon , +}_{\alpha , \beta , q} (u),$$
and  by (\ref{5.1}) we obtain
$$ E^{\varepsilon , +}_{\alpha , \beta , q} (\overline{u} ) <  m^{g,+}_{\alpha , \beta , \ep}  + 2 \delta_0.$$
\end{proof}

\begin{remark}\label{remark5}
On any closed Riemannian manifold for any $\varepsilon >0$  there is a set of points $x_j$, $j=1,...,K_{\varepsilon}$ such that
the closed balls $\overline{   B_g(x_j , \varepsilon )   }$ are disjoint, and the set is maximal under this condition. It follows that the balls $\overline{ B_g(x_j , 2 \varepsilon ) }$ cover $M$.
It is easy to construct closed sets $A_j$, $ \overline {B_g(x_j ,  \varepsilon) } \subset A_j \subset  
\overline{ B_g(x_j , 2 \varepsilon ) }$ which cover $M$ and only intersect in their
boundaries. Moreover one can see by a volume comparison argument that if $\varepsilon$ is small enough there is a constant $K$, independent of
$\varepsilon$,  such that for any point in $M$ belongs to at most $K$ of the balls $B_g(x_j , 3\varepsilon )$.
($K$ would be approximately the maximal number of disjoint  balls of radius $1/4$ in the Euclidean
space  contained in the ball of radius $1$).
\end{remark}

The next result shows that for $\ep$ small there is a small ball which contains a meaningful part of
$u^+$. It is similar to the results in \cite[Lemma 5.3]{Benci} and  \cite[Proposition 3.6]{Petean} for a similar
situation for (second order) Yamabe-type equations. 
\begin{proposition} There exist $\gamma >0$ and $\varepsilon_0 >0$ such that for any
$\varepsilon \in (0,\varepsilon_0 )$ if $u\in  \mathcal{N}_g^+ (\alpha , \beta , q, \ep)$ then 
there exists $x \in M$ such that 
$$\varepsilon^{-n} \int_{B_g(x, 2 \varepsilon )} (u^+ )^{q+1} dv_g  \geq \gamma. $$
\end{proposition}
\begin{proof}
Let us consider $\varepsilon_0 >0$ such that for $\varepsilon < \varepsilon_0$ one can construct sets like in the previous remark. 

Consider $u\in \mathcal{N}_g^+ (\alpha , \beta , q, \ep)$. 
Let $u_j = u^+   \chi_{A_j}$ be the
restriction of $u^+$ to $A_j$ (extended by 0 away from $A_j$ in order to consider it as a function on $M$). Then
\begin{equation}\label{equ512}\varepsilon^{-n} \int_M  \big(  \ep^4 ( \Delta_g u )^2  +
 \beta \ep^2   | \nabla  u |^2  +\alpha u^2 \big)\, dv_g
 \begin{array}[t]{ll}
 \displaystyle = \varepsilon^{-n}
  \int_M (u^+ )^{q+1} dv_g \medskip\\
  \displaystyle 
  =\sum_{j=1}^{K_{\ep}}   \  ( \varepsilon^{- \frac{n(q-1)}{q+1}  }     \| u_j \|_{q+1}^{q-1}  ) \  (    \varepsilon^{-\frac{2n}{q+1}}   \| u_j \|^2_{q+1} )\medskip\\
  \displaystyle 
  \leq \left( \max_j   {\varepsilon^{-\frac{n(q-1)}{q+1}}}   {   \| u_j \|_{{q+1}}^{q-1} } \right)
\sum_{j}   \  { \varepsilon^{-\frac{2n}{q+1}}}   { \| u_j \|^2_{q+1 } }   .
\end{array}
\end{equation} 
Let
$$\Gamma =  \max_j   {\varepsilon^{-\frac{n(q-1)}{q+1}}}   {   \| u_j \|_{q+1}^{q-1} }.$$
Note that we have to prove that there exists a positive constant $\gamma$, independent of $\ep$ and $u$ such that 
$$\Gamma^{\frac{q+1}{q-1}}  \geq \gamma .$$
To obtain such a lower bound for $\Gamma$ we   have to bound from above each of the 
terms $ \varepsilon^{-\frac{2n}{q+1} }   \| u_j \|_{{q+1}}^2  $.  

We first point out that from the local, second order, Sobolev inequalities, there is a constant
$C$ independent of $\ep$ and $x$ such that for any $x\in M$, $v\in H_0^2 (B_g(x,\ep ) )$,
$$
\ep^{\frac{-2n}{q+1} } \left(  \int_{B_g(x,\ep )}  |v|^{q+1} dv_g  \right)^{\frac{2}{q+1}} \leq  C \ep^{-n} \int_{B_g(x,\ep )} 
  \big(\ep^4 ( \Delta_g v )^2  +
 \beta \ep^2   | \nabla  v |^2  +\alpha v^2 \big)\,dv_g .
$$

Now, let $\varphi_{\varepsilon}$ be a cut-off function on $\re^n$ which is 1 in $B_g(0, 2 \varepsilon )$, vanishes away from
$B_g(0, 3 \varepsilon )$, $| \nabla \varphi_{\varepsilon} | \leq  3/\varepsilon$ and
$| \Delta \varphi_{\varepsilon} |  \leq 9/\ep^2 $ in the intermediate annulus.

Define, for $j=1 \dots K_{\varepsilon}$,
$$ u_{j, \varepsilon}  (x) = u  (x)\, \varphi_{\varepsilon} (d(x, x_j )).$$

Since $u_j \leq  | u_{j, \varepsilon} | $ we have that 
\begin{equation}\label{Jloc0}
\varepsilon^{-\frac{2n}{q+1}} \| u_j \|_{{q+1}}^2 \leq   \varepsilon^{-\frac{2n}{q+1}} \| u_{j , \varepsilon}   \|_{{q+1}}^2
\end{equation}
and by using the Sobolev inequalities we obtain that, for some constant $C$,
\begin{equation}\label{Jloc}  \varepsilon^{-\frac{2n}{q+1}} \| u_{j , \varepsilon}   \|_{{q+1}}^2  \leq C \varepsilon^{-n}  \int_{B_g(x_j,3\ep )} 
 \big( \ep^4 ( \Delta_g u_{j,\ep} )^2  +
 \beta \ep^2   | \nabla  u_{j,\ep} |^2  +\alpha u_{j,\ep}^2 \big) \,dv_g  .
 \end{equation}
Since $ | u_{j, \varepsilon} |  \leq | u | $ we have that
$$\varepsilon^2 | \nabla u_{j,\varepsilon}  |^2  \leq 2\varepsilon^2 | \nabla u |^2 + 18 (u )^2 $$
and
\begin{equation*}
\ep^4 (\Delta_g u_{j , \varepsilon}  )^2 \leq 3 \ep^4 (\Delta_g u )^2 + 108 \ep^2  | \nabla u |^2 
+243 u^2 .
\end{equation*}
 Since $\alpha$, $\beta$ are fixed positive constants and any
point $x\in M$ could be in at most $K$ of the balls $B_g(x_j ,3\ep )$ 
(and therefore, it has an open neighbourhood which can only intersect the support of the
corresponding $\nabla u_{j,\varepsilon} $'s), it follows from the previous inequalities that 
there is a constant $E$ (independent of $\ep$) such that 
\begin{equation}\label{Jloc2}
 \sum_j \int_{B_g(x_j ,3\ep )} 
\big(\ep^4 (\Delta_g u_{j , \varepsilon}  )^2 + \ep^2  \beta | \nabla u_{j , \varepsilon}  |^2
+\alpha  u_{j , \varepsilon} ^2\big)\, dv_g \leq E \int_M \big(\ep^4 (\Delta_g u)^2 +\ep^2 \beta
| \nabla u |^2 + \alpha u^2\big)\, dv_g .
\end{equation}
Then, from (\ref{equ512}), (\ref{Jloc0}), (\ref{Jloc}) and (\ref{Jloc2}), we get 
\begin{equation*}
\varepsilon^{-n} \int_M\big(  \ep^4 (\Delta_g u)^2 +\ep^2 \beta
| \nabla u |^2 + \alpha u^2\big)\, dv_g   \leq \Gamma C  E 
\varepsilon^{-n} \int_M \big(  \ep^4 (\Delta_g u)^2 +\ep^2 \beta
| \nabla u |^2 + \alpha u^2 \big)\,dv_g 
\end{equation*}

Therefore
$$\Gamma \geq \frac{1}{ C E}$$
and we can take
$$\gamma =  {  \left( \frac{1}{C E} \right) }^{\frac{q+1}{q-1} }  .$$
\end{proof}
Finally we can prove Theorem \ref{theorem5.2}
\begin{proof}[Proof of Theorem   \ref{theorem5.2}] Assume the theorem is  not true.  Then there exist  $\eta <1$ and  sequences of positive numbers  $\varepsilon_j \rightarrow 0$, $\delta_j \rightarrow 0$, and
$u_j \in \Sigma^{\alpha , \beta}_{\varepsilon_j , {m}^{g,+}_{\alpha , \beta , \varepsilon_j } + \delta_j }$
such that for any
$x\in M$ one has 
$$
\int_{B_g(x,r)}   (u_j^+ )^{q+1} < \eta \int_M (u_j^+ )^{q+1} .
$$
For each $j$ large enough the previous Proposition  provides  $x_j \in M$ such that, for some fixed $\gamma >0$,
$$ 
\varepsilon_j^{-n} \int_{B_g(x_j , 2 \varepsilon_j )}  (u_j^+ )^{q+1} \  dv_g  \geq \gamma .
$$
Lemma \ref{lema67}  then gives a function $\overline{u_j} = u_{j,1} + u_{j,2}$  such that $\overline{u_j} \in  
 \Sigma^{\alpha , \beta}_{\varepsilon_j , {m}^{g,+}_{\alpha , \beta , \varepsilon_j } + 2 \delta_j }$,
$ u_{j,1}$ is supported inside the ball of radius $r$ centered at $x_j$, $ u_{j,1}$ and $ u_{j,2}$ have disjoint support, $ u_j  =  u_{j,1 }$ in $B_g(x_j , \varepsilon_j )$ and $u_j = u_{j,2}$
outside $B_g(x_j , r )$.

We have that
$$ 
\varepsilon_j^{-n} \int_{M}  (u_{j,1}^+ )^{q+1} \  dv_g  \geq \gamma 
$$
and, using (\ref{equivalencias}),
$$ 
\varepsilon_j^{-n} \int_{M}  (u_{j,2}^+ )^{q+1} \  dv_g >  \varepsilon_j^{-n} (1 - \eta )   \int_M (u_j^+ )^{q+1} \  dv_g    \geq (1- \eta )  \frac{2(q+1)}{q-1} {m}^{g,+}_{\alpha , \beta ,  \varepsilon_j }  .
$$
Then  it follows from Corollary \ref{coro65}  that there exists $\delta >0$, independent of $j$, such that
$$
E^{\varepsilon_j ,+}_{\alpha , \beta ,q} (\overline{u_j}  ) \geq  \Psi (\delta ) 
{m}^{g,+}_{\alpha , \beta , \varepsilon}.
$$
But for $j$ large enough we have that 
$$
E^{\varepsilon_j ,+}_{\alpha , \beta ,q} (\overline{u_j}  )  < 
{m}^{g,+}_{\alpha , \beta , \varepsilon_j }  + 2\delta_j  
< \Psi (\delta ) {m}^{g,+}_{\alpha , \beta , \varepsilon_j} ,
$$
reaching a contradiction.
\end{proof}

\section{The Lusternik-Schnirelmann category of $\Sigma_{\ep, {\bf m}_{\ep} + \delta} $}\label{section7}

In this section we   prove:
\begin{theorem}\label{theorem7} There exist $\ep_0 >0$,   $\delta_0 >0$ such that
for any $\ep \in (0, \ep_0 )$,  we have
$$\mathrm{Cat}(\Sigma_{\ep, {\bf m}_{\ep} + \delta_0 } ) \geq \mathrm{Cat}(M).$$
\end{theorem}
Note that then Theorem \ref{mainThm} follows from Corollary \ref{cor36} .

In the proof we   make use of the center of mass function defined in \cite[Section 5]{Petean}. We  
briefly recall the definition. For $r>0$, $\eta \in (0,1)$ let
$$L^1_{r,\eta} = \bigg\{ u \in L^1 (M) \setminus \{ 0 \} : \sup_{x\in M} \int_{B_g (x,r)} |u| \,dv_g  > \eta  \int_M |u|\, dv_g \bigg\} .$$
If $r_0 >0$ is such that for any $x\in M$ the ball $B_g(x,r_0 )$ is strongly convex then it is proved in 
\cite[Theorem 5.2]{Petean} for any $r < r_0 /2$ and any $\eta >1/2$ that there exist a continuous function
${\bf CM} (r,\mu ): L^1_{r,\mu} \rightarrow M$ such that for any  $x \in M$ such that
$\int_{B_g(x,r)} |u| dv_g  > \eta  \int_M |u| dv_g$ we have ${\bf CM} (r,\eta ) (u)  \in B_g(x,2r)$. 

\begin{proof}[Proof of Theorem \ref{theorem7}] Fix $r < r_0 /2$. Let $\eta =0.9$. Fix $\ep_0 >0$, $\delta_0 >0$ as in
Theorem \ref{theorem5.1} and pick any $\delta 
\in (0, \delta_0 )$ . Then for any $\ep \in (0, \ep_0 )$ we have that if 
$u \in \Sigma_{\varepsilon , {\bf m}_{\varepsilon} + \delta} $
then $ (u^+ )^{q+1} \in L^1_{r,\eta}$. 

\medskip

Define ${\bf b} : \Sigma_{\varepsilon , {\bf m}_{\varepsilon} + \delta} \rightarrow M$ by
${\bf b} (u) = {\bf CM} (r,0.9 )  \left(  (u^+ )^{q+1} \right)$.  Note that ${\bf b}$ is a continuous function.

\medskip

Next we consider for any $x\in M$ the function $ {\overline{U}}^{\ep , r}_{\alpha , \beta , x}$ as
in Remark \ref{remark4}, which is supported in $B_g(x,r)$. Consider as in Definition \ref{definition4.4} the function
${\bf i}(x) \in \mathcal{N}(\alpha , \beta , q, \ep)$.

Note that there exists $\lambda >0$ such that $ \lambda {\bf i}(x) \in 
\mathcal{N}^g_{\ep}$, with $\mathcal{N}^g_{\ep}$ as in (\ref{N_ep}), and
$\lambda $ is close to 1 by Remark \ref{remark1}. We can also see from Remark \ref{remark1}, Proposition \ref{propo53}, Theorem \ref{theorem4}, Theorem \ref{theorem5}  that with $\delta$ fixed we can choose $\ep_0$ smaller if necessary
so that for any $\ep \in (0,\ep_0 )$ we have that 
$ J^g_{\ep} ( \lambda {\bf i}(x) ) < {\bf m}_{\ep}  + \delta$.\smallskip

Define ${\bf  p}:M \rightarrow \Sigma_{\varepsilon , {\bf m}_{\varepsilon} + \delta}$
by ${\bf p} (x) =\lambda  {\bf i}(x)$.\smallskip 

It follows from Theorem 5.2  in \cite{Petean}  that for any $x$, ${\bf b} ({\bf p} (x)) \in
B_g(x, 2r)$. Then if follows from the election of $r$ that  for any $x\in M$, ${\bf b}
\circ {\bf p} (x)  $ is joined to $x$ by a unique minimizing geodesic contained in 
$B_g(x,2r)$. Then ${\bf b}
\circ {\bf p}$ is  homotopic to the identity of $M$ and the theorem follows.
\end{proof}

\end{document}